\documentclass{article}

\usepackage[utf8]{inputenc}
\usepackage{amsfonts} 
\usepackage{graphicx}
\usepackage{url}
\usepackage[normalem]{ulem} 
\usepackage{tikz}
\usepackage{color}
\usepackage{subcaption}
\usetikzlibrary{cd}
\tikzset{commutative diagrams/diagrams={ampersand replacement=\&}}
\usepackage{overpic}
\usepackage{float}
\usepackage[T1]{fontenc}
\usepackage[framed,numbered]{matlab-prettifier}

\usepackage{relsize}
\usepackage{amsmath,amscd, amssymb}
 \usepackage[margin=1in]{geometry}
\usepackage{mathtools}
\usepackage{tikz}
\usetikzlibrary{calc} 
\usetikzlibrary{shapes,arrows,positioning,decorations.pathreplacing,decorations.pathmorphing}

\definecolor{darkblue}{rgb}{0.0, 0.0, 0.8}
\definecolor{darkred}{rgb}{0.8, 0.0, 0.0}
\definecolor{darkgreen}{rgb}{0.0, 0.8, 0.0}

\usepackage{amsthm}
\usepackage{thmtools}
 \theoremstyle{definition}
 \newtheorem{theorem}{Theorem}[section] 
 \newtheorem*{theorem*}{Theorem}

 \newtheorem{lemma}[theorem]{Lemma}

 \newtheorem{corollary}[theorem]{Corollary}

 \newtheorem{proposition}[theorem]{Proposition}

 \newtheorem{example}[theorem]{Example}
 \newtheorem{remark}[theorem]{Remark}
 \newtheorem{definition}[theorem]{Definition}

\usepackage{hyperref}

\usepackage{cleveref}
\crefname{figure}{Fig.}{Figs.}
\crefname{equation}{Eqn.}{Eqns.}
\crefname{theorem}{Thm.}{Thms.}
\crefname{proposition}{Prop.}{Props.}
\crefname{section}{Sec.}{Secs.}
\crefname{lemma}{Lem.}{Lems.}
\crefname{definition}{Defn.}{Defns.}
\crefname{corollary}{Cor.}{Cors.}
\crefname{remark}{Rem.}{Rems.}
\crefname{example}{Ex.}{Exs.}

\newcommand{\K}{\mathbb{K}}

\newcommand{\R}{\mathbb{R}}

\newcommand{\Z}{\mathbb{Z}}

\newcommand{\e}{\mathbf{\varepsilon}}

\renewcommand{\AA}{\mathcal{A}}

\newcommand{\FF}{\mathcal{F}}

\newcommand{\PP}{\mathcal{P}}

\newcommand{\Kfunc}{K}

\newcommand{\Cfunc}{C}

\newcommand{\Mfunc}{M}
\newcommand{\Nfunc}{N}

\newcommand{\Vect}{\mathbf{Vect}}

\newcommand{\vZero}{\textbf{0}}

\newcommand{\transition}[2]{\varphi^{#2}_{#1}}
\newcommand{\barc}{B}
\newcommand{\id}{\mathrm{Id}}
\newcommand{\field}{\mathbb{K}}

\newcommand{\TODO}[3]{\hbox to 0pt{\textcolor{#1}{$^\bullet$}}\marginpar{\footnotesize \textcolor{#1}{\begin{flushleft}#2: #3\end{flushleft}}}}

\title{On the Hausdorff stability of barcodes over posets}
 \author{Mujtaba Ali, Tom Needham, Anastasios Stefanou, Ling Zhou}

\begin{document}

\maketitle
\begin{abstract}
The Isometry Theorem of Chazal et al.~and Lesnick is a fundamental result in persistence theory, which states that the interleaving distance between two one-parameter persistence modules is equal to the bottleneck distance between their barcodes. Significant effort has been devoted to extending this result to modules defined over more general posets. As these modules do not generally admit nice decompositions, one must restrict attention to the class of interval-decomposable modules in order to define an appropriate notion of bottleneck distance. Even with this assumption, it is known that bottleneck distance may not be equivalent to interleaving distance, but that it is Lipschitz stable under certain, fairly restrictive, assumptions. In this paper, we consider the more basic question of stability of the Hausdorff distance with respect to interleaving distance for interval-decomposable modules. Our main theorem is a Lipschitz stability result, which holds in a fairly general setting of interval-decomposable modules over arbitrary posets, where intervals are assumed to be taken from any family satisfying certain closure conditions. Along the way, we develop some new tools and results for interval-decomposable modules over arbitrary posets, in the form of geometrically-flavored characterizations of the existence of morphisms and interleavings between interval modules.

\end{abstract}

\section{Introduction}
 
The interleaving distance plays a central role in topological data analysis, because it provides a canonical and functorial way to compare persistence modules arising from data, such as persistent homology modules~\cite{chazal2009proximity,carlsson2009topology}. 
Unlike distances defined directly on barcodes or other discrete invariants of modules, the interleaving distance is defined at the level of persistence modules themselves and is therefore independent of any chosen decomposition. This makes it applicable in full generality, in particular in multiparameter settings where complete discrete invariants are unavailable. Most importantly, the interleaving distance is tightly linked to stability: persistent homology is \(1\)-Lipschitz with respect to natural perturbations of the input data, such as the \(\ell^\infty\)-distance on filtrations, and the interleaving distance on the resulting persistence modules. For this reason, the interleaving distance is widely regarded as the theoretically correct notion of distance in persistence theory and serves as a benchmark against which the stability of other invariants and metrics is measured.

Despite these conceptual advantages, the interleaving distance is not an ideal metric on barcodes from a computational perspective. In particular, computing the interleaving distance is NP-hard even when restricted to interval-decomposable modules \cite{bjerkevik2019computing}. 
Nevertheless, any alternative metric on barcodes that is used in practice should remain compatible with interleaving-distance--based stability results. In particular, any meaningful metric on barcodes arising from the persistent homology of multifiltrations should be Lipschitz stable with respect to the interleaving distance.

This requirement follows from the optimality of the interleaving distance \cite{lesnick2015theory}. First, persistent homology is \(1\)-Lipschitz with respect to the \(\ell^\infty\)-distance on filtrations and the interleaving distance on homology modules. Second, any other distance on persistence modules enjoying this $1$-Lipschitz stability property with persistence homology must necessarily be a lower bound for the interleaving distance. Hence, while alternative barcode metrics may offer improved computational tractability, they can only be stable in a meaningful sense if they are controlled by the interleaving distance.

From this perspective, a central open problem in multiparameter persistence is the identification of a distance on barcodes that is polynomial-time computable from the input data and Lipschitz stable with respect to the interleaving distance. The bottleneck distance is known to be equal to the interleaving distance in the one-parameter setting~\cite{lesnick2015theory}, but its generalization to modules defined over more general posets requires one to restrict attention to interval-decomposable modules (as modules defined over arbitrary posets do not necessarily admit interval decompositions; see~\cite{botnan2022introduction}). Even working under this assumption, the generalized bottleneck distance (see Defn.~\ref{def:bottleneck}) is known to be unstable with respect to the interleaving distance for modules defined over the poset $\R^n$, when $n \geq 2$ \cite{botnan2018algebraic}. 
However, two prominent conjectures address the possibility of recovering stability under additional structural assumptions on the intervals. 
The \emph{Botnan--Lesnick conjecture} predicts Lipschitz stability of the barcode for interval-decomposable \(\mathbb{R}^n\)-modules whose intervals are geometrically convex  \cite{botnan2018algebraic,bjerkevik2016stability}, while the \emph{Bjerkevik conjecture} \cite{bjerkevik2025stabilizing} proposes the same conclusion when the intervals are required to be uppersets. We observe that both conjectures concern important classes of \(\mathcal F\)-decomposable modules, where \(\mathcal F\) denotes a chosen family of intervals closed under finite intersections.

Motivated by the issues and conjectures described above, we consider a more basic question. In the setting of interval-decomposable modules, one can define a generalized notion of Hausdorff distance between modules, which can be viewed as a relaxation of the bottleneck distance (see Defn.~\ref{def:bottleneck}). The main result of our paper is as follows. The statement uses $d_\mathrm{H}$ for Hausdorff distance and $d_\mathrm{I}$ for interleaving distance, which requires the notion of an \emph{flow} to define over a general poset---see Sec.~\ref{sec:flows}.

\begin{theorem*}
Let $(\PP,\Omega)$ be a poset equipped with a $\R$-flow and let $\FF$ be a family of intervals in $\PP$ that contains the empty interval, is closed under intersections, and is closed under the action of $\Omega$. 
Let $\Mfunc,\Nfunc:\PP\to\Vect_K$ be two interval-decomposable $\PP$-modules with barcodes $\barc(\Mfunc),\barc(\Nfunc)\subset \FF$. 
Then
\[d_\mathrm{H}(\Mfunc,\Nfunc)\leq 2\cdot d_\mathrm{I}(\Mfunc,\Nfunc).\]
\end{theorem*}

To contextualize this result, let us compare to a theorem of Bjerkevik~\cite{bakke2021stability}, which says that if $\FF$ is the family of rectangles in $\R^n$ and $B(M),B(N) \subset \FF$, then $d_\mathrm{B}(M,N) \leq (2n-1) \cdot d_\mathrm{I}(M,N)$, where $d_\mathrm{B}$ denotes bottleneck distance. While our result uses the weaker metric $d_\mathrm{H}$, it applies to more general posets and families of intervals, and involves an absolute (e.g., dimension-independent) Lipschitz constant. 

A main motivation for the work in this paper was to develop an approach to the Botnan-Lesnick and Bjerkevik conjectures described above. The difficulty in proving these conjectures stems from the inherently global nature of the bottleneck distance. Stability requires constructing matchings between entire intervals that are coherent across all parameters, rather than merely pointwise or locally. 
By contrast, local notions such as the \emph{matching distance}~\cite{cerri2013betti,kerber2018exact} on fiber barcodes are more tractable, since they reduce to one-parameter persistence by restricting to lines of positive slope in \(\R^n\), where canonical optimal matchings are available. However, for such local matchings to induce a global matching of interval decompositions—and hence a bottleneck bound—one must be able to lift a compatible family of matchings across all fibers to a single global matching of intervals. This coherence problem is highly nontrivial and represents a fundamental obstruction to establishing bottleneck stability with respect to the interleaving distance. From this perspective, establishing Hausdorff stability for interval-decomposable modules over general posets offers a more flexible and geometrically natural alternative. 

In the one-parameter case, all modules are intersection-closed, and the local equality of Hausdorff and bottleneck distances was first observed by Chazal et al.~\cite{chazal2009proximity}.
This local equivalence plays a central role in the seminal stability result of Chazal et al., where bottleneck stability is obtained by combining Hausdorff control with the fact that the interleaving distance defines a geodesic metric on the space of persistence modules \cite{chazal2009proximity}. In future work, we plan to apply our main theorem to extend this approach to attack bottleneck stability results for more general poset modules.

Finally, besides being of  theoretical interest, establishing stability results for barcodes of convex-interval-decomposable modules and for upper-set-decomposable modules  is particularly important because of concrete applications. For instance, the articles~\cite{xin2023gril,dey2025quasi} recently introduced the notion of \emph{worms}
---barcodes whose intervals consist of certain geometrically-convex regions in $\mathbb{R}^n$, arising from a certain generalization of the one-parameter \emph{persistent landscape} \cite{bubenik2015statistical}, in their \emph{GRIL algorithm}.~Upperset-decomposable modules also appear naturally in several settings in persistence theory and relative homological algebra; cf.~\cite{bjerkevik2025stabilizing,asashiba2023approximation,blanchette2024homological,botnan2024signed,botnan2024bottleneck,chacholski2025koszul,oudot2024stability,kim2024interleaving}. 
Thus, in these settings, Lipschitz stability of barcodes with respect to a suitable metric, relative to perturbations measured by the interleaving distance, guarantees the robustness of the resulting invariants proposed in these works.

\section{Modules over a poset}
\label{sec:Modules Over a Poset}

This section describes concepts related to poset theory and persistence modules over a poset. The material here is standard, and this section mostly serves to set notation. For more background in this area, consult, e.g.,~\cite{bubenik2015metrics,bjerkevik2016stability,miller2019modules}.

\subsection{Posets and intervals}

We will refer to any partially ordered set $\PP=(\PP,\leq)$, as a \emph{poset}. We consider the poset $\PP$ as a category whose objects are elements of $\PP$ and whose morphisms are described by the partial order: there exists a (unique) morphism from object $p$ to object $q$ if and only if $p \leq q$.

\begin{definition}[{Interval in a Poset}]
\label{def:interval}
A subset $I \subset \PP$ is called
\begin{enumerate}
    \item\label{item:convexity} \emph{poset-convex} if,
    for each pair of points $p,q \in I$ with $p \leq q$, any $r \in \PP$ such that $p \leq r \leq q$ must satisfy $r \in I$;
    \item\label{item:connectedness} \emph{poset-connected} (or, simply, \emph{connected}) if, for each pair of points $p,q \in I$, there exists a \emph{path} from $p$ to $q$ in $I$---that is, there exists a sequence of points $r_1,\ldots,r_{2n} \in I$ such that 
    \[
    p \leq r_1 \geq r_2 \leq r_3 \geq \cdots \leq r_{2n-1} \geq r_{2n} \leq q.
    \]
\end{enumerate}
If $I$ is both poset-convex and poset-connected, then we say that $I$ is an \emph{interval}.
\end{definition}

\begin{remark}
    By definition, the empty set is an interval of $\PP$. We will refer to this as the \emph{empty interval}.
\end{remark}

One should not confuse poset-connectedness with connectedness in the topological sense. However, we frequently borrow topological terminology. For example, a subset $U \subset \PP$ is called a \emph{(poset-)connected component} of $\PP$ if it is poset-connected and maximal with respect to inclusion: if $V \subset P$ is connected and $V \cap U \neq \varnothing$ then $V \subset U$. 

\begin{example} \label{ex:R_n_poset}
    The poset $(\R^n,\leq_n)$ of $n$-tuples of real numbers is equipped with the coordinatewise partial order
    \begin{equation}\label{eqn:poset_structure}
    (t_1,t_2,\ldots,t_d) \leq_n (s_1,s_2,\ldots,s_d) \Leftrightarrow t_i \leq s_i,\text{ for any }i=1,\ldots,n.
    \end{equation}
    For simplicity, we will often write $\leq$ for $\leq_n$ when the meaning is clear.

    We note that a line segment with nonzero finite slope in $\R^2$ fails to be an interval. Specifically, segments with positive slope violate poset-convexity, while those with negative slope violate poset-connectedness. See \cref{fig:examples_intervals} 
    for illustrations of these phenomena and examples of intervals and non-intervals in $\R^2$ (see \cref{fig:examples_convex} for additional examples).
    
    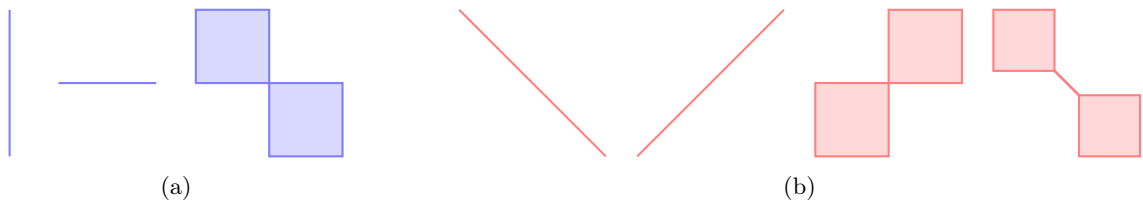
\begin{figure}[H]
\centering

\begin{subfigure}{.44\textwidth}
\centering
    \begin{tikzpicture}[scale=0.65] 
    \draw[thick, blue!50] (0,0) -- (0,3);
    \draw[thick, blue!50] (1,1.5) -- (3,1.5);
    \end{tikzpicture}
    \hspace{0.75em}
    \begin{tikzpicture}[scale=0.65] 
    \draw[thick, blue!50, fill=blue!15] (0,1.5) -- (0,3) -- (1.5,3) -- (1.5,1.5) -- (3,1.5) -- (3,0) -- (1.5,0) -- (1.5,1.5) -- cycle;
    \end{tikzpicture}
\caption{}
\label{fig:line_intervals}
\end{subfigure}
\begin{subfigure}{.55\textwidth}
\centering

    \begin{tikzpicture}[scale=0.65] 
    \draw[thick, red!50] (3,0) -- (0,3);
    \end{tikzpicture}
    \hspace{0.5em}
    \begin{tikzpicture}[scale=0.65] 
    \draw[thick, red!50] (0,0) -- (3,3);
    \end{tikzpicture}
    \hspace{0.5em}
    \begin{tikzpicture}[scale=0.65] 
    \draw[thick, red!50, fill=red!15] (0,1.5) -- (0,0) -- (1.5,0) -- (1.5,1.5) -- (3,1.5) -- (3,3) -- (1.5,3) -- (1.5,1.5) -- cycle;
    \end{tikzpicture}
    \hspace{0.5em}
    \begin{tikzpicture}[scale=0.65] 
    \draw[thick, red!50, fill=red!15] (0,1.75) -- (0,3) -- (1.25,3) -- (1.25,1.75) -- (1.75,1.25) -- (3,1.25) -- (3,0) -- (1.75, 0) -- (1.75,1.25)  -- (1.25,1.75) -- cycle;
    \end{tikzpicture}
\caption{}
\label{fig:line_non_intervals}
\end{subfigure}

\caption{Examples in $\mathbb{R}^2$: (a) intervals, shown in blue; (b) non-intervals, shown in red.}

\label{fig:examples_intervals}
\end{figure}

\end{example}

We now present a simple but useful structural result on intervals in posets.
While every sub-poset admits a unique partition into connected components, we show that if the sub-poset satisfies the poset convexity axiom (i.e., \cref{item:convexity} of \cref{def:interval}), then its connected components are themselves intervals. These are called the \emph{interval components} of the sub-poset.

\begin{lemma}\label{lem:interval-components}
Let $\PP$ be a poset and let $U \subset \PP$ be poset-convex.
Then every connected component of $U$ is an interval.
Consequently, $U$ admits a unique partition
\(
U = \bigsqcup_{a \in \AA} Q_a,
\)
where $\AA$ is an indexing set and each $Q_a$ is a non-empty interval component of $U$.
\end{lemma}

\begin{proof}
Let $Q$ be a connected component of $U$. Since $Q$ is poset-connected by assumption, we need only to show that $Q$ is poset-convex.
Take $p,q\in Q\subset U$ and $r\in \PP$ with $p\le r\le q$.  
Since $U$ is poset-convex, $r\in U$.  
Because $Q$ is poset-connected, for any point $t \in Q$, there exists a path from $t$ to $p$, and hence a path from $t$ to $r$ by concatenating with $p\leq r$. We thus obtain
paths in $Q\cup \{r\}$ from every point of $Q$ to $r$, so the maximality of $Q$ implies $r \in Q$, and it follows that $Q$ is poset-convex. Since $Q$ is poset-convex and poset-connected, it is an interval.
\end{proof}

In what follows, we will frequently need to consider properties of intersections of subsets of a poset. The following result will be useful.

\begin{lemma}\label{lem:intersection of poset-convex}
    Let $I,J \subset \PP$ be poset-convex subsets. Then $I \cap J$ is also poset-convex. It follows that the interval components of $I \cap J$ are well-defined.
\end{lemma}

\begin{proof}
    Let $p,q \in I \cap J$ with $p \leq q$ and let $r \in \PP$ with $p \leq r \leq q$. Then poset-convexity of $I$ and $J$ implies that $r \in I \cap J$ as well. The last sentence of the lemma then follows from Lem.~\ref{lem:interval-components}.
\end{proof}

A \emph{poset morphism} is an order-preserving map between posets.  
A poset morphism is a \emph{poset isomorphism} if it is bijective and its
inverse is also order-preserving.  
It is straightforward to check that poset isomorphisms preserve poset-convexity and poset-connectedness, and therefore preserve interval structure.

\subsection{Modules over a poset}

Throughout this section, fix a poset $\PP$. 
We also fix a field of coefficients $\field$ and let $\Vect_{\field}$ denote the category of vector spaces over $\field$.

\begin{definition}[$\PP$-Module]\label{def:P_Modules}
A \emph{$\PP$-module} is a functor $\Mfunc:\PP\to \Vect_{\field}$.
Specifically, $\Mfunc$ consists of
\begin{itemize}
    \item a $\PP$-indexed family of $\field$-vector spaces $\Mfunc_p$, $p\in \PP$,
    \item for every inequality $p\leq q$ in $\PP$, a $\field$-linear transformation $\varphi_\Mfunc(p,q):\Mfunc_p\to\Mfunc_q$,  such that 
    \[\varphi_\Mfunc(q,r)\circ\varphi_\Mfunc(p,q)=\varphi_\Mfunc(p,r),\]
    whenever $p\leq q\leq r$ in $\PP$, and $\varphi_\Mfunc(p,p)$ is the identity map on $\Mfunc_p$, for any $p\in\PP$.
\end{itemize}

A \emph{morphism $f:\Mfunc\to \Nfunc$ of $\PP$-modules} is a natural transformation $f$ from the functor $\Mfunc$ to the functor $\Nfunc$.
Specifically, $f$ consists of a $\PP$-indexed family of linear transformations $f_p:\Mfunc_p\to\Nfunc_p$, $p\in P$, such that for every $p\leq q$ the diagram 
    \[
    \begin{tikzcd}
 \Mfunc_p\arrow[d, swap, "{\varphi_\Mfunc(p,q)}"]\arrow[r, "{f_p}"]\&\Nfunc_p\arrow[d, "{\varphi_\Nfunc(p,q)}"]\\
\Mfunc_q\arrow[r, "{f_q}"]\&\Nfunc_q
\end{tikzcd}
\]
commutes.

If $f:\Mfunc \to \Nfunc$ and $g:\Nfunc \to \Kfunc$ are morphisms of modules over $P$, then \emph{the composition $g\circ f: \Mfunc \to \Kfunc$ of $f$ with $g$} is the morphism defined pointwise by
\[(g\circ f)_p\coloneqq g_p\circ f_p\text{ for all }p\in \PP.\]
It is not hard to show that if each linear transformation $f_p$ is a linear  isomorphism, then $f$ is an \emph{isomorphism from $M$ to $N$}; that is, it admits and inverse $f^{-1}$ defined by $(f^{-1})_p = (f_p)^{-1}$, such that $f \circ f^{-1}$ and $f^{-1} \circ f$ are identity functors. If there exists an isomorphism $f:M\to N$, we write $M\cong N$.

We will say that the morphisms $f:M\to N$ and $f':M'\to N'$ are equivalent, denoted $f\cong f'$, if there are isomorphisms $g:M'\to M$ and $h:N\to N'$ such that $f'=h\circ f\circ g$.
\end{definition}

\subsection{Interval decomposable modules}
We are primarily concerned with persistence modules which have a special decomposition structure. We give the necessary definitions below.

\begin{definition}[Characteristic Module]
For $I$ an interval in $\PP$, the \emph{characteristic $\PP$-module} of $I$,  
$\Cfunc(I):\PP\to\Vect_{\field}$, is defined as
\[ 
\Cfunc(I)_p\coloneqq\begin{cases} 
      \field, & \text{ if }p\in I \\
      \vZero, & \text{ otherwise}
   \end{cases}
\qquad \mbox{ and } \qquad 
\varphi_{\Cfunc(I)}(p,q)\coloneqq
\begin{cases} 
      \id_{\field}, & \text{ if }p,q\in I \\
      \vZero, & \text{ otherwise,}
   \end{cases}
\]
where $\mathbf{0}$ denotes the zero vector space, or the zero linear map, respectively, and $\id_{\field}$ is the identity map on the base field $\field$.
A persistence module $\Mfunc:(P,\leq)\to\Vect_{\field}$ is an \emph{interval module} if $\Mfunc=\Cfunc(I)$, for some interval $I$ of $P$. In this case, $I$ is the \emph{support of the interval module $\Mfunc$}. For $I=\varnothing$, one has $C(\varnothing)=\mathbf{0}$.
\end{definition}

\begin{definition}[Interval Decomposable Module]
A $\PP$-module $\Mfunc:\PP\to\Vect_{\field}$ is \emph{interval-decomposable} if 
\[\Mfunc\cong\bigoplus_{I\in  \barc(\Mfunc)}\Cfunc(I),\]
for a multiset\footnote{By a multiset, $B(M)$, we mean a multiset representation $\{(I_a,a)\mid a\in A\}$ whose base set $\{I_a\mid a\in A\}$ consists of intervals of $\PP$; see \cite{bauer2020persistence}.} of nonempty intervals in $\PP$, $\barc(\Mfunc)$, called the \emph{barcode} of $\Mfunc$.
\end{definition}

The barcode is uniquely determined (this follows from the Azumaya–Krull–Remak–Schmidt module decomposition theorem~\cite{azumaya1950corrections}). 

Let $\Mfunc,\Nfunc$ be interval-decomposable $\PP$-modules.
Any morphism $f:\Mfunc\to\Nfunc$ is equivalent (in the sense of \cref{def:P_Modules}) to a direct sum of morphisms of interval modules, i.e.
\begin{equation}\label{eqn:morphism_decomposition}
    f \cong \bigoplus_{(I,J)\in  \barc(\Mfunc)\times  \barc(\Nfunc)}f_{I,J},
\end{equation}
where $f_{I,J}:\Cfunc(I)\to \Cfunc(J)$ is a morphism of  interval modules. If $f:M\to N$ and $g:N\to L$ are morphisms of interval modules, then their composition morphism $g\circ f:M\to L$ has
\begin{equation}\label{eqn:composition morphism}
(g\circ f)_{I,K}=\sum_{J\in  \barc(\Nfunc)}g_{J,K}\circ f_{I,J},
\end{equation}
for any pair $(I,K)\in B(M)\times B(L)$. Hence, the study of morphisms of interval-decomposable modules reduces to the study of morphisms of interval modules.

\section{Morphisms of interval modules}

We now restrict attention to interval modules. These constitute the basic building blocks of interval-decomposable modules, and understanding their morphisms is essential for the analysis of interleavings and distance notions developed later in this work. In this section, we study morphisms between interval modules from a geometric perspective, emphasizing how their existence and behavior are controlled by the relative position of the underlying intervals in the poset. The results of this section have appeared previously in the context of $\R^n$ modules~\cite{dey2018computing}, but, to our knowledge, have not been systematically considered in the TDA literature for more general poset modules.

\subsection{Geometric characterization of morphisms of interval modules}

In this subsection, we give a geometric characterization of morphisms between interval modules over $\PP$, generalizing work of Dey and Xin~\cite{dey2018computing} in the context of modules over $\R^n$. We begin by defining a property of a $\PP$-indexed family of linear maps.

\begin{definition}[Consistent Family of Maps]
    Let $I,J$ be intervals in $\PP$ with interval components of $I\cap J$ denoted $\{Q_a\}_{a \in \mathcal{A}}$ (see Lem.~\ref{lem:intersection of poset-convex}). A family $f = \{f_p:C(I)_p \to C(J)_p \mid p \in \PP\}$ of linear maps is called \emph{consistent} if
    \begin{enumerate}
        \item for every interval component $Q_a$ and every pair $p,q \in Q_a$, we have $f_p = f_q$, and
        \item for every $p\notin I\cap J$,  $f_p$ is the zero map.
    \end{enumerate}  
    In this case, there exists a scalar $\omega_a(f)\in\field$, such that for any $p\in Q_a$, the linear map $f_p:\Cfunc(I)_p\to\Cfunc(J)_p$ acts by multiplication by $\omega_a(f)$ on $\Cfunc(I)_p = \field = \Cfunc(J)_p$.
\end{definition}

The next proposition shows that consistency of the underlying family of maps is a necessary condition for any morphism of interval modules.

\begin{proposition}
\label{prop:necessary condition}
Let $I,J$ be intervals in $\PP$ and let  $f: \Cfunc(I) \to\Cfunc(J)$ be a morphism. Then $f = \{f_p\}_{p \in \PP}$ forms a consistent family.
\end{proposition}

 \begin{proof}
If $I \cap J = \emptyset$, then the claim follows trivially, so assume not. 
Since $f \colon \Cfunc(I) \to \Cfunc(J)$ is a morphism, naturality gives
\[
\varphi_{\Cfunc(J)}(p,q) \circ f_p
=
f_q \circ \varphi_{\Cfunc(I)}(p,q)
\qquad
\text{for all } p \le q \text{ in } Q_a .
\]
As both structure maps $\varphi_{\Cfunc(I)}(p,q)$ and $\varphi_{\Cfunc(J)}(p,q)$ act as the identity on $\field$,
it follows that $f_p = f_q$ for all $p \le q$ in $Q_a$.
Hence, for $p\in Q_a$, each $f_p$ is a linear endomorphism of $\field$, and therefore given by multiplication by a scalar
in $\field$.
By the poset-connectedness of the interval $Q_a$, this scalar is independent of $p\in Q_a$, and thus constant
throughout $Q_a$. 

On the other hand, suppose $p \notin I\cap J$. Because either $p \notin I$ or $p \notin J$, we have $\Cfunc(I)_p=0$ or $\Cfunc(J)_p=0$. Thus, the linear map $f_p:\Cfunc(I)_p\to\Cfunc(J)_p$ is the zero map.
 \end{proof}

Next, we introduce a purely poset-theoretic property, with a view toward a geometric condition which is equivalent  to the converse of Prop.~\ref{prop:necessary condition}. 

\begin{definition}[{Valid Component~\cite[Defn.~15]{dey2018computing}}]
    Let $I,J \in \PP$ be intervals and let $Q \subset I \cap J$ be a connected component. We say that $Q$ is \emph{$(I,J)$-valid} if the following hold for every $p \in Q$: 
    \begin{enumerate}
        \item for all $q \in I$ such that $q \leq p$, we have $q \in J$, and 
        \item for all $r \in J$ such that $r \geq p$, we have $r \in I$.
    \end{enumerate}
\end{definition}

The following result gives a generalization of \cite[Prop.~16]{dey2018computing}; the proof is the same, and we include it here for convenience of the reader.

\begin{proposition}
\label{prop:equivalent condition}
Let $I,J$ be intervals in $\PP$, let $\{Q_a\}_{a \in \mathcal{A}}$ denote the interval components of $I \cap J$, and let $f = \{f_p\}_{p \in \PP}$ be a consistent family of linear maps, with associated scalars denoted $\omega_a(f)$. The following are equivalent:
\begin{enumerate}
    \item $f$ defines a morphism $C(I) \to C(J)$;
    \item for every $Q_a$ such that $\omega_a(f) \neq 0$, $Q_a$ is $(I,J)$-valid.
\end{enumerate}
\end{proposition}

\begin{proof}
    First, suppose that $f$ defines a morphism and let $Q_a$ be a connected component such that $\omega_a(f) \neq 0$. Let $p \in Q_a$, $q \in I$, and suppose $q \leq p$. By naturality, we have 
    \[
    0 \neq f_p \circ \varphi_{\Cfunc(I)}(q,p) = \varphi_{\Cfunc(J)}(q,p) \circ f_q,
    \]
    hence $q \in J$. The other condition for validity is proved similarly.

    We now prove the converse, which requires us to show that each diagram of the form
        \[
    \begin{tikzcd}
 C(I)_q\arrow[d, swap, "{f_q}"]\arrow[rr, "{\varphi_{\Cfunc(I)}(q,p)}"]\& \&C(I)_p \arrow[d, "{f_p}"]\\
C(J)_q\arrow[rr, "{\varphi_{\Cfunc(J)}(q,p)}"]\&\& C(J)_p
\end{tikzcd}
\]
commutes. Let us consider the only interesting case, where at least one of the maps is nonzero, say $f_p \circ \varphi_{\Cfunc(I)}(q,p)$ (the other case is proved by a similar argument). Then $\varphi_{\Cfunc(I)}(q,p)$ and $f_p $ are both  nonzero, which immediately implies $q,p \in I$. Moreover, consistency of $f$  implies $p \in I \cap J$; in particular, $p$ lies in some path component $Q_a$, so that $f_p \circ \varphi_{\Cfunc(I)}(q,p) = f_p \circ \mathrm{Id}_\mathbb{K}$ is multiplication by the nonzero scalar $\omega_a(f)$. By the validity assumption, $q \in I$ and $q \leq p$ together imply that $q \in J$, hence $q \in I \cap J$. As $q \leq p$, $q$ must also lie in the path component $Q_a$, so that $\varphi_{\Cfunc(J)}(q,p) \circ f_q = \mathrm{Id}_\mathbb{K} \circ f_q$ is also multiplication by $\omega_a(f)$. 
\end{proof}

The proposition immediately yields the following  corollary, which we use repeatedly throughout the paper.

 \begin{corollary}
 \label{cor:nonzeromorphismiffvalid}
 Let $I,J$ be intervals in $\PP$ such that $I\cap J$ is an  interval.
 Then, there exists a nonzero
 morphism $f:\Cfunc(I) \to\Cfunc(J)$ if and only if $I\cap J$ is $(I,J)$-valid.
 \end{corollary}

\subsection{Geometric characterization of nonzero compositions with interval support}

We provide below a criterion for checking whether a composition of two nonzero morphisms of interval modules is nonzero, when the underlying support of the intervals satisfy certain conditions. 
This will be useful to us in the following sections.

\begin{proposition}
\label{prop:nonzero composition}
Let $I,J,K$ be intervals in $\PP$ such that $I\cap J$, $J\cap K$, and $I\cap K$ are intervals.
Let $f:\Cfunc(I) \to \Cfunc(J)$ and $g:\Cfunc(J) \to \Cfunc(K)$ be two nonzero morphisms.
Then, the composition morphism $g\circ f$ is nonzero
if and only if $\varnothing \neq I\cap K\subset J$.
\end{proposition}

\begin{proof} 
Because $f$ and $g$ are nonzero and $I\cap J$, $J\cap K$ are intervals,
\cref{prop:necessary condition} gives
\[
f_p\neq \vZero \iff p\in I\cap J,
\qquad
g_p\neq \vZero \iff p\in J\cap K.
\]
Thus $(g\circ f)_p\neq \vZero$ can occur only when both $f_p$ and $g_p$ are nonzero,  
so only when $p\in I\cap J\cap K$.
That is, 
\[
\{p\in\PP \mid (g\circ f)_p\neq \vZero\}
\subseteq  I\cap J\cap K.
\]

If $g\circ f\neq \vZero$, since $I\cap K$ is an interval, 
\cref{prop:necessary condition} implies
\(
\varnothing \neq \{p\in\PP \mid (g\circ f)_p\neq \vZero\} = I\cap K.
\)
Thus, we have $I\cap K \subset I\cap J\cap K$, which is equivalent to $I\cap K \subset J$. Conversely, assume $\varnothing\neq I\cap K\subset J$. For any $p\in I\cap K$ we
then have $p\in I\cap J$ and $p\in J\cap K$, so by
\cref{prop:necessary condition} both $f_p$ and $g_p$ are nonzero scalar maps.
Therefore,
\(
(g\circ f)_p = g_p\circ f_p \neq \vZero,
\)
and hence $g\circ f\neq \vZero$.
\qedhere
\end{proof}

\begin{remark}
Note that if $g\circ f \neq 0$, then necessarily $f \neq 0$, $g \neq 0$, and $I \cap J \neq \varnothing$. One might ask whether the converse holds, that is, whether the additional condition $I \cap K \subset J$ is superfluous for ensuring that $g \circ f$ is nonzero under these assumptions. This is not the case. Indeed, there exist examples where $f \neq 0$, $g \neq 0$, and $I \cap K \neq \varnothing$, yet $g \circ f = 0$ unless the extra relation $I \cap K \subset J$ is imposed. Such an example is illustrated in \cref{fig:comp_exp}.
\end{remark}
\usetikzlibrary{patterns}
\begin{figure}[H]
\centering
\begin{tikzpicture}[scale=0.4, line cap=round, line join=round]

\def\xmin{-6} \def\xmax{7}
\def\ymin{-4} \def\ymax{5}

\def\rxL{-6} \def\rxR{-1}
\def\ryB{1}  \def\ryT{5}

\coordinate (L1a) at (\xmin,2);
\coordinate (L1b) at (\xmax,-1);

\coordinate (L2a) at (-2.2,\ymax);  
\coordinate (L2b) at (1.2,\ymin);   

\draw[thick, fill=black!10] (\rxL,\ryB) rectangle (\rxR,\ryT);

\begin{scope}
  \clip (\xmin,\ymin) rectangle (\xmax,\ymax);
  \clip (L2a) -- (L2b) -- (\xmax,\ymin) -- (\xmax,\ymax) -- cycle;
  \fill[pattern=horizontal lines, pattern color=black!35]
       (\xmin,\ymin) rectangle (\xmax,\ymax);
\end{scope}

\begin{scope}
  \clip (\xmin,\ymin) rectangle (\xmax,\ymax);
  \clip (L1a) -- (L1b) -- (\xmax,\ymin) -- (\xmin,\ymin) -- cycle;
  \fill[pattern=vertical lines, pattern color=black!35]
       (\xmin,\ymin) rectangle (\xmax,\ymax);
\end{scope}

\draw[thick] (L1a) -- (L1b);
\draw[thick] (L2a) -- (L2b);

\node at (-3.5,3.4) {\large J};
\node at (5,3.2) {\large I};
\node at (-3.2,-2.2) {\large K};

\end{tikzpicture}
\caption{An example of $g\circ f=0$, where $0\neq f: C(I)\to C(J)$ and $0 \ne g: C(J)\to C(K)$.
Here, each $f_p$ is the identity map on $I\cap J$ and zero elsewhere, and each $g_p$ is the identity map on $J\cap K$ and zero elsewhere, but their composition $g_p\circ f_p$ is always zero.}
\label{fig:comp_exp}
\end{figure}
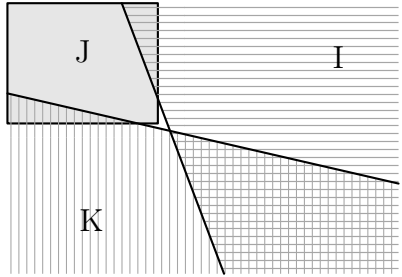

\section{Distances on interval-decomposable modules}
\label{sec:Distances}
We now return to interval-decomposable $\PP$-modules. In this context, we recall the definitions of $\R$-flows on a poset, and the interleaving-, bottleneck-, and Hausdorff distances. These distances are known to be extended pseudometrics\footnote{A function satisfying the axioms of a metric but taking values in $[0,\infty]$ is called an \emph{extended metric}, and if in addition to that, distinct points are allowed to have distance zero, it is called an \emph{extended pseudometric}.}; cf.~\cite{chazal2009proximity,edelsbrunner2010computational} on $\PP$-modules.

\subsection{Flows on posets}\label{sec:flows}

In this section, define the notion of interleaving distance between $\PP$-modules. In the context of $\R$-modules, this definition goes back to~\cite{chazal2009proximity}. It has since become a fixture in the literature on interleaving distances, and has been widely generalized~\cite{bubenik2014categorification,bubenik2015metrics,deSilva2018,scoccola2020locally,mcfaddin2026interleaving}. Four our purposes, it will be most natural to work in the \emph{category with a flow framework} of~\cite{deSilva2018}.

We first define the concept of an $\R$-flow on a poset. For our purposes, we consider flows as $\mathbb{R}$-actions, rather than $[0,\infty)$-actions, on posets, which is the relevant setting for the applications we have in mind. 

\begin{definition}[$\R$-Flow]
Let $\PP$ be a poset. 
An \emph{$\R$-flow} on $\PP$ is a family of poset maps 
$\Omega = \{\Omega_t: \mathcal P \to \mathcal P\}_{t\in\R}$, satisfying
\begin{enumerate}
\item $\Omega_t(p)\leq \Omega_s(p)$, for all $t\leq s$ in $\R$ and $p \in \PP$,
    \item $\Omega_{t+s} = \Omega_t \circ \Omega_s$, for all $t,s \in \R$,
    \item $\Omega_0 = \mathrm{id}_{\PP}$.
\end{enumerate}
\end{definition}

\begin{example}\label{ex:R_flow}
The main paradigm of a poset with $\R$-flow is the poset of $n$-tuples of real numbers with diagonal shifts.
Recall from \cref{eqn:poset_structure} the poset structure on $\R^n$.
For every $t\in\R$, the $t$-action $\Omega_t$ on $\R^n$ is defined as 
\begin{align*}
    \Omega_t:\R^n&\to\R^n\\
    (a_1,a_2,\ldots,a_n)&\mapsto (a_1+t,a_2+t,\ldots,a_n+t) \\
    (a_1,a_2,\ldots,a_n) \leq (b_1,b_2,\ldots,b_n)&\mapsto (a_1+t,a_2+t,\ldots,a_n+t) \leq (b_1+t,b_2+t,\ldots,b_n+t).
\end{align*}
We then easily check that $\Omega\coloneqq(\Omega_t)_{t\in\R}$ forms an $\R$-flow on $\R^n$.
\end{example}

We define the notion of \emph{$t$-shift of a subset $I$ of $\PP$} as 
\[I(t)\coloneqq \{p\in\PP \mid \Omega_{t}(p)\in I\} = \Omega_{-t}(I).\]

\begin{example}
The collection of all nonempty subsets of a poset~$\PP$, ordered by inclusion, together with the action of~$\Omega_{-t}$ on these subsets for all~$t \in \mathbb{R}$, provides another example of a poset with an~$\mathbb{R}$-flow.
\end{example}

The following propositions establish some useful properties of $\R$-flows.

\begin{proposition}
\label{prop:poset iso}
If $\Omega$ is an $\R$-flow, then for every $t\in\R$, the functor $\Omega_t:\PP\to\PP$ is a poset isomorphism.
\end{proposition}
\begin{proof}
Fix $t\in\R$. Since $\R$ is a group, $-t\in\R$, and by the $\R$-flow law
\[
\Omega_t\circ\Omega_{-t}=\Omega_{t+(-t)}=\Omega_0=\id_\PP,
\qquad
\Omega_{-t}\circ\Omega_t=\id_\PP.
\]
Thus, $\Omega_t:\PP\to\PP$ is bijective with the inverse, $\Omega_{-t}$.

If $x\le y$, then order-preservation of $\Omega_t$ gives $\Omega_t(x)\le\Omega_t(y)$.
Conversely, if $\Omega_t(x)\le\Omega_t(y)$, then applying the order-preserving map 
$\Omega_{-t}$ yields
\[
x=\Omega_{-t}(\Omega_t(x)) \le \Omega_{-t}(\Omega_t(y))=y.
\]
Hence $x\le y \iff \Omega_t(x)\le \Omega_t(y)$, so $\Omega_t$ is an poset isomorphism.
\end{proof}

\begin{proposition}
Let $I\subset \PP$. Then, $I$ is an interval in $\PP$ if and only if $I(t)$ is an interval in $\PP$.
\end{proposition}

\begin{proof}
By \cref{prop:poset iso}, each map $\Omega_t:\PP\to\PP$ is a poset isomorphism.
Since $I(t)=\Omega_{-t}(I)$ and $I=\Omega_t(I(t))$, the claim follows from the
fact that poset isomorphisms preserve poset-convexity and poset-connectedness,
and therefore preserve interval structure.
\end{proof}

\begin{proposition}\label{prop:shift invariance}
Let $I,J$ be intervals in $\PP$ and $t\in\R$. Then, $I\cap J$ is an interval if and only if $I(t)\cap J(t)$ is an interval. 
\end{proposition}

\begin{proof}
    By \cref{def:interval}, $I\cap J$ is an interval if and only if it is poset-convex and poset-connected. Since the poset structure is preserved under shifts, $I\cap J$ is poset-convex (resp. poset-connected) if and only if $I(t)\cap J(t)$ is poset-convex (resp. poset-connected). Thus, $I\cap J$ is an interval if and only if $I(t)\cap J(t)$ is an interval.
\end{proof}

For the rest of this section, we fix a poset with an $\R$-flow $(\PP,\Omega)$. 

\subsection{Interleaving distance}
We define shift functors and transition morphisms on $\PP$-modules, induced by the $\R$-flow $\Omega$ on $\PP$, and then we define the associated interleaving distance on $\PP$-modules.

\begin{definition}[Shift Functors and Transition Morphisms]
For $t\in \R$, we define the \emph{shift functor} $(\cdot)(t): \Vect_K^{\PP}\to\Vect_K^{\PP}$ as follows: For
$M$ a $\PP$-module we let $M(t)$ be the $\PP$-module such that for all $p\in\PP$ we have
$M(t)_p\coloneqq M_{\Omega_t(p)}$, and for all $p\leq q$ in $\PP$ we have 
$\varphi_{M(t)}(p,q)\coloneqq \varphi_M(\Omega_t(p),\Omega_t(q))$. 
For a morphism $f$ in $\Vect_K^{\PP}$, we define $f(t)$ by taking $f(t)_{p}\coloneqq f_{\Omega_t(p)}$. 

For $M$ a $\PP$-module and $\delta\geq 0$, let the \emph{$\delta$-transition morphism}
$\varphi^{\delta}_M: M \to M(\delta)$ be the morphism whose restriction to $M_p$ is the $\mathbb{K}$-linear map $\varphi_M(p,\Omega_\delta(p))$, for all
$p\in\PP$.

If $M$ is interval-decomposable, then the $\delta$-transition morphism $\transition{\Mfunc}{\delta}$ of $M$ is
\begin{equation}\label{eq:transition matrix form}
    \transition{\Mfunc}{\delta}=\bigoplus_{I,I' \in  \barc(M)} \delta_{I,I'}\cdot\varphi_{C(I)}^{\delta}, \qquad \mbox{where} \qquad 
\delta_{I,I'}\coloneqq \left\{\begin{array}{rl}
1_{\field} & \text{if }I = I' \\
0_{\field} & \text{otherwise .} \end{array}\right.
\end{equation}
Here, $1_{\field}$ and $0_{\field}$ denote the identity and zero elements of the field $\field$, respectively.
\end{definition}

We now record some basic properties the shift functor. The proof of the following is immediate.

\begin{lemma}
\label{lem:invariant properties}
The shift functor has the following properties:
    \begin{enumerate}
\item for any $t\in\R$ and any interval $I$ in $\PP$, one has $(\Cfunc(I))(t)=\Cfunc(I(t))$,
    \item for any morphism $f:\Cfunc(I)\to \Cfunc(J)$ of interval $\PP$-modules and any $t\in\R$, the shifted morphism $f(t)$ is nonzero if and only if $f$ is nonzero,
    \item for any interval-decomposable $\PP$-module $M$ and any $t\in\R$, the barcode satisfies $B(M(t))=\{I(t)\mid I\in B(M)\}$,
    \item for any morphism $f\cong\bigoplus f_{I,J}$ of interval-decomposable $\PP$-modules $M$ and $N$, and any $t\in\R$, we have 
    \[
    (f(t))_{I(t),J(t)} \cong (f_{I,J})(t).
    \]
    \end{enumerate}
\end{lemma}

Finally, we introduce the notion of interleaving distance between $\PP$-modules.

\begin{definition}[{Interleaving Distance~\cite{bubenik2015metrics}}]
\label{dfn:interleaving distance}
 Let $\Mfunc,\Nfunc$ be a pair of $\PP$-modules.
\begin{itemize}
    \item We say that $M,N$ are \emph{$\e$-interleaved} if there exists a pair of morphisms $f:\Mfunc\to\Nfunc(\e)$ and $g:\Nfunc\to\Mfunc(\e)$ such that the triangles \begin{equation*}
	\begin{tikzcd}
	\Mfunc
	\arrow[dr,  "{\transition{\Mfunc}{2\e}}"']
	\arrow[r,  swap, "{f}"']
	\&
	\Nfunc(\e)
	\arrow[d,  crossing over,  "{g(\e)}"]
	\&
	\&
	\Nfunc
	\arrow[r,  swap, "{g}"']
	\arrow[dr,    "{\transition{\Nfunc}{2\e}}"']
	\&
	\Mfunc(\e) 
	\arrow[d,    "{f(\e)}"]
	\\
	\&
	\Mfunc(2\e)
	\&
	\&
	\&
	\Nfunc(2\e)
	\end{tikzcd}
	\end{equation*}
	commute.
	\item We say that the ordered pair $(\Mfunc ,\Nfunc)$ are \emph{left (right) $\e$-interleaved} if there exists a pair of morphisms $f:\Mfunc \to \Nfunc(\e)$ and $g:\Nfunc \to \Mfunc(\e)$ which satisfy the left (right) diagram.
	\item 	The \emph{interleaving distance} between $\Mfunc$ and $\Nfunc$ is defined as 
	\[d_{\mathrm{I}}(\Mfunc,\Nfunc)\coloneqq\inf\{\e\geq0|\text{ }\Mfunc,\Nfunc\text{ are }\e\text{-interleaved}\}\] 
	If $\Mfunc,\Nfunc$ are not $\e$-interleaved for any $\e\geq0$, we define their interleaving distance to be $\infty$.
	\end{itemize}
\end{definition}

\subsection{Hausdorff and bottleneck distances}

We now recall the notions of Hausdorff and bottleneck distances between interval-decomposable $\PP$-modules. The Hausdorff distance is a ubiquitious tool in geometry and topology, and our definition is an instance of the usual notion (see, e.g.,~\cite{hausdorff1978grundzuge,burago2001course}). The bottleneck distance first appeared in the context of one-parameter persistent homology (i.e., for $\R$-modules)---see~\cite{frosini2001size,cohen2007stability}---and has since been generalized to the version that we consider below~\cite{bjerkevik2016stability}.

Throughout this subsection, fix such modules $M$ and $N$. Recall that we consider their barcodes $B(M)$ and $B(N)$ as multisets.

\begin{definition}[$\e$-Correspondence, Hausdorff and Bottleneck Distances]\label{def:bottleneck}
    Let $\e > 0$. A sub-multiset $\sigma \subset B(M) \times B(N)$ is called an \emph{$\e$-correspondence} if 
    \begin{itemize}
        \item for each $(I,J) \in \sigma$, $C(I)$ and $C(J)$ are $\e$-interleaved;
        \item if $I \in B(M)$ does not belong to any ordered pair in $\sigma$, then $C(I)$ is $\e$-interleaved with the zero module;
        \item if $J \in B(N)$ does not belong to any ordered pair in $\sigma$, then $C(J)$ is $\e$-interleaved with the zero module,
    \end{itemize}
    where the above are all considered with  multiplicity. We define the \emph{Hausdorff distance} between $M$ and $N$ as 
    \[
    d_\mathrm{H}(M,N) \coloneqq \inf \{\e \geq 0 \mid \mbox{there exists an $\e$-correspondence between $M$ and $N$}\}.
    \]
    
    We say that $\sigma$ is an \emph{$\e$-matching} if it satisfies the conditions above, as well as 
    \begin{itemize}
        \item for each $I \in B(M)$, there is at most one $J \in B(N)$ such that $(I,J) \in \sigma$;
        \item for each $J \in B(N)$, there is at most one $I \in B(M)$ such that $(I,J) \in \sigma$,
    \end{itemize}
    where these are once again understood to hold with multiplicity. The \emph{bottleneck distance} between $M$ and $N$ is then defined as 
    \[
    d_\mathrm{B}(M,N) \coloneqq \inf \{\e \geq 0 \mid \mbox{there exists an $\e$-matching between $M$ and $N$}\}.
    \]
\end{definition}

\begin{remark}
\label{rem:hausdorff remk}
    Consider the metric space $(X,d_\mathrm{I})$, where $X$ is the set of interval modules over $\PP$. 
    With $\vZero$ denoting the zero module, one can check that our definition of Hausdorff distance in Defn.~\ref{def:bottleneck}  is equivalent to the usual notion of Hausdorff distance (induced by $(X,d_\mathrm{I})$) between the sets 
    \[
    \{C(I) \mid I \in B(M)\} \cup \{\mathbf{0}\} \quad \mbox{and} \quad \{C(J) \mid J \in B(N)\} \cup \{\mathbf{0}\};
    \]
    here, we really mean the underlying \emph{sets}, rather than multisets. See also \cite[Lem.~3.1]{memoli2012some}.
\end{remark}

As $\e$-matchings are special cases of $\e$-correspondences, one immediately has the following bound.

\begin{proposition}
\label{prop:hausdorff bottleneck bound}
For any pair of interval-decomposable $\PP$-modules $M,N$, we have
\[
d_{\mathrm{H}}(M,N)\leq d_\mathrm{B}(M,N).
\]
\end{proposition}

\section{Intersection-closed families of intervals}
\label{sec:The Geometry of Convex Interval Modules}
We restrict our attention to interval-decomposable $\PP$-modules whose 
constituent intervals lie in a family \(\FF\) of intervals that is closed under 
intersections. Within this setting, we characterize existence of interleavings between 
interval modules in purely geometric terms—namely, via the structure of their 
interval intersections.

Although our results here can be extended to general interval-decomposable $\PP$-modules—by working 
with the interval components of intersections of arbitrary intervals—we focus on 
the intersection-closed setting, as this is the class most relevant for our 
intended study of interleavings. 
For corresponding characterizations of interleavings 
in the full interval-decomposable setting specifically over the poset \(\PP = \mathbb{R}^n\), we refer 
to \cite{dey2018computing}.

\subsection{Intersection-closed families of intervals}
In general, an intersection of a pair of intervals in a poset $\PP$ is either the empty set or a disjoint union of intervals. However, there are families of intervals such that the intersection of any finite collection of intervals is again an interval or the empty set.

\begin{definition}[Properties of Interval Families]
A family $\FF$ of intervals in $\PP$  is said to be:
\begin{itemize}
    \item \emph{intersection-closed}, if for every pair of intervals $I,J\in \FF$,  $I\cap J$ is an interval in $\FF$. 
    \item \emph{closed under the action of $\Omega$}, if for any interval $I$ in $\FF$ and any $t\in\R$, $\Omega_t(I)\in \FF$.
\end{itemize}
\end{definition}

\begin{example}\label{ex:intersection_closed}
    Examples of intersection-closed families of intervals in $\PP$ are
    \begin{itemize}
        \item the family of uppersets in $\PP$,
        \item the family of downsets in $\PP$,
        \item the family of intervals, if we assume $\PP$ is a totally ordered set,
                   \item the family of geometrically convex intervals in $\R^n$, for $n\geq2$ (defined precisely below),
                \item the family of intervals in the sub-poset $\mathbb{U}:=\{(a,b)\in \R^{op}\times \R\mid a\leq b\}$ of $\R^{op}\times \R$,
                 \item the family of intervals in the sub-poset $\mathbb{ZZ}:=\{(a,b) \in \Z^{op}\times \Z\mid a = b\text{ or }a =b +1\}$ of $\mathbb{Z}^{op}\times \Z$,
    \item the family of rectangles in $\Z^n$, for $n\geq2$.
    \end{itemize}  
        One can check, that the first five families above are also closed under the action of the respective poset-$\R$-flows on their intervals. 
        If we take the $\FF\coloneqq\mathbf{Int}(\PP)$ to be the family of all intervals in $\PP$, then this is not always intersection-closed and not always closed under the action of the poset-$\R$-flow $\Omega$ on intervals (e.g.~for $\PP\coloneqq\R^2$).  
\end{example}

\begin{figure}[H]
\centering

\begin{subfigure}{.47\textwidth}
\centering
    \begin{tikzpicture}[scale=0.65] 
    \fill[blue!15](0,0) rectangle (3,3);
    \end{tikzpicture}
    \hspace{0.55em}
    \begin{tikzpicture}[scale=0.65] 
    \fill[blue!15] (0,2) -- (0,3) -- (1,3) -- (3,1) -- (3,0) -- (2,0) -- cycle;
    \end{tikzpicture}
    \begin{tikzpicture}[scale=0.65] 
    \fill[fill=blue!15]
   (-0.5,3) 
   .. controls (2,2.8) and (3,2) .. (3,0)
   .. controls (0,0) and (-0.5,1) .. (-0.5,3);
    \end{tikzpicture}
\caption{}
\label{fig:2D_convex}
\end{subfigure}
\begin{subfigure}{.12\textwidth}
\centering
    \begin{tikzpicture}[scale=0.65] 
    \fill[blue!25] (0,1) -- (0,3) -- (2,3) -- (2,2) -- (3,2) -- (3,0) -- (1,0) -- (1,1) -- cycle;
    \end{tikzpicture}
\caption{}
\label{fig:2D_non_convex}
\end{subfigure}
\begin{subfigure}{.31\textwidth}
\centering
    \begin{tikzpicture}[scale=0.65] 
    \fill[red!15] (3,3) -- (3, 2) -- (1,0) -- (0,0) -- (0,1) -- (2,3) -- cycle;
    \end{tikzpicture}
    \begin{tikzpicture}[scale=0.65] 
    \fill[red!15]
   (0.5,3) 
   .. controls (-2,2.8) and (-3,2) .. (-3,0)
   .. controls (0,0) and (0.5,1) .. (0.5,3);
    \end{tikzpicture}
\caption{}
\label{fig:2D_non_interval}
\end{subfigure}%

\caption{(a) Geometrically convex intervals in $\R^2$. (b) An interval in $\R^2$ that is not geometrically convex. (c) Geometrically convex subsets of $\R^2$ that are not intervals.}

\label{fig:examples_convex}
\end{figure}
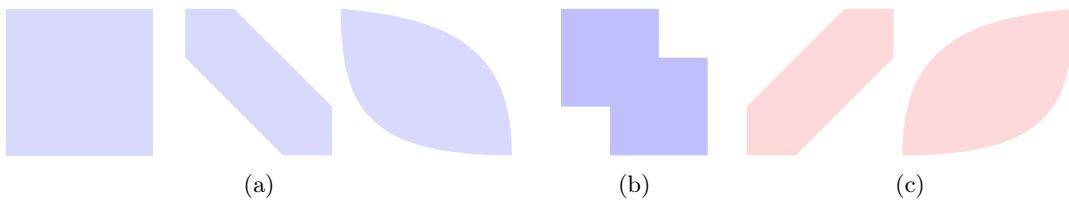

We now focus on a particularly natural family of intervals. By a \emph{(geometrically) convex interval} in $\R^n$, we mean a  poset interval in $\R^n$, which is topologically open and geometrically convex as subset of $\R^n$.

\begin{proposition}\label{prop:convex_intersection}
Let $I\subset\R^n$ be nonempty, open, and geometrically convex. Then $I$ is
poset-connected. 
In particular, if $I$ and $J$ are convex intervals in
$\R^n$, then their intersection $I\cap J$ is either empty or a convex interval.
\end{proposition}

\begin{proof}
Fix $p,q\in I$. If $p$ and $q$ are comparable, then $p\le q$ or $q\le p$, so the one-step path $p\le q$ (or $q\le p$) lies in $I$ and we are done. Assume henceforth that $p$ and $q$ are incomparable.

The line segment $\overline{pq}$ is contained in $I$ (by geometric convexity) and is compact. For each $x\in\overline{pq}$, let $\rho(x)\coloneqq\operatorname{dist}_\infty(x,\R^n\setminus I)$ be the $\ell^\infty$-distance from $x$ to the complement of $I$. Since $\rho$ is continuous and positive on the compact set $\overline{pq}$, it attains a positive minimum $\e>0$ (by the extreme value theorem).
Thus, for every $x\in\overline{pq}$, the open ball $B_\infty(x,\e)$ is contained in $I$. 

Partition $\overline{pq}$ into $K$ equal-length subsegments with endpoints $p=p_0,p_1,\ldots,p_K=q$, where $K$ is chosen large enough that $d_\infty(p_{i-1},p_i)<\e$ for all $i=1,\ldots,K$. 
Then each ball $B_\infty(p_i,\e)$ contains both $p_{i-1}$ and $p_{i}$, and hence the line segment $\overline{p_{i-1}p_{i}}$. 
Note that every $\ell^\infty$–ball is poset-connected: for any
$a,b$ in such a ball, their coordinatewise maximum also lies in the ball,
yielding a two–step path $a\le \max(a,b)\ge b$.
Thus each consecutive pair $p_{i-1},p_i$ may be joined by a path
lying in $I$, and concatenating these paths produces a
path from $p$ to $q$ in $I$.

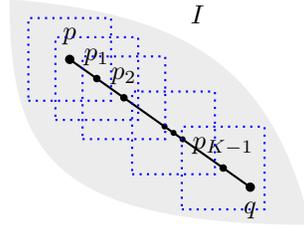
\begin{figure}[H]
\centering
\begin{tikzpicture}[scale=1]

\fill[fill=gray!15]
   (-0.5,3) 
   .. controls (2,2.8) and (3,2) .. (3.5,0)
   .. controls (0,0) and (-0.5,1) .. (-0.5,3);

\coordinate (p) at (0.3,2.2);
\coordinate (q) at (2.7,0.5);

\node[fill=black,circle,inner sep=1.25pt,label=above:$p$] at (p) {};
\node[fill=black,circle,inner sep=1.25pt,label=below:$q$] at (q) {};

\draw[thick] (p) -- (q);

\def\size{1.1}
\draw[blue,dotted,thick] ($(p)+(-\size/2,-\size/2)$) -- ($(p)+(\size/2,-\size/2)$) -- ($(p)+(\size/2,\size/2)$) -- ($(p)+(-\size/2,\size/2)$) -- cycle;

\coordinate (p1) at ($(p)!0.15!(q)$);
\coordinate (p2) at ($(p)!0.3!(q)$);
\coordinate (p3) at ($(p)!0.85!(q)$);
\coordinate (p4) at ($(p)!0.575!(q)$);

\node[circle,fill=black,inner sep=1pt,label=above:$p_1$] at (p1) {};
\node[circle,fill=black,inner sep=1pt,label=above:$p_2$] at (p2) {};
\node[circle,fill=black,inner sep=1pt,label=above:$p_{K-1}$] at (p3) {};

\node[circle,fill=black,inner sep=0.8pt] at ($(p)!0.525!(q)$) {};
\node[circle,fill=black,inner sep=0.8pt] at (p4) {};
\node[circle,fill=black,inner sep=0.8pt] at ($(p)!0.625!(q)$) {};

\draw[blue,dotted,thick] ($(p1)+(-\size/2,-\size/2)$) -- ($(p1)+(\size/2,-\size/2)$) -- ($(p1)+(\size/2,\size/2)$) -- ($(p1)+(-\size/2,\size/2)$) -- cycle;
\draw[blue,dotted,thick] ($(p2)+(-\size/2,-\size/2)$) -- ($(p2)+(\size/2,-\size/2)$) -- ($(p2)+(\size/2,\size/2)$) -- ($(p2)+(-\size/2,\size/2)$) -- cycle;
\draw[blue,dotted,thick] ($(p3)+(-\size/2,-\size/2)$) -- ($(p3)+(\size/2,-\size/2)$) -- ($(p3)+(\size/2,\size/2)$) -- ($(p3)+(-\size/2,\size/2)$) -- cycle;
\draw[blue,dotted,thick] ($(p4)+(-\size/2,-\size/2)$) -- ($(p4)+(\size/2,-\size/2)$) -- ($(p4)+(\size/2,\size/2)$) -- ($(p4)+(-\size/2,\size/2)$) -- cycle;

\node at (2,2.8) {$I$};

\end{tikzpicture}
\caption{An open and geometrically convex $I$, two points $p,q\in I$, subdivision points $p_i$ on the line $\overline{pq}$, and a pruned chain of open $\ell_\infty$-balls $B_\infty(p_i,\e)\subset I$ with overlapping intersections along the segment.}
\label{fig:convep_open_set}
\end{figure}

Next, let $I,J$ be convex intervals in $\R^n$. 
If $ I\cap J = \varnothing$, we are done.
Otherwise, the intersection $I\cap J$ is open and geometrically convex, as both $I$ and $J$ are.
By the above, $I\cap J$ is poset-connected. 
Since $I\cap J$ is also poset-convex, it is a convex interval. 
\end{proof}

\begin{remark}
    A priori, convex intervals are open, poset-convex,
poset-connected, and geometrically convex subsets of $\R^n$.  
However, by
\cref{prop:convex_intersection}, the poset-connectedness requirement is redundant.
Thus, convex intervals are precisely those open subsets of $\R^n$ that are
geometrically and poset-convex.

    We note that the openness requirement is essential for ensuring the intersection-closure property of convex intervals. 
    Without it, two closed convex intervals (such as the one shown in the middle of \cref{fig:2D_convex}) can intersect in a line segment of nonzero finite slope, which fails to be an interval in $\R^2$ (see \cref{ex:R_n_poset}).
\end{remark}

We have the following immediate corollary (verifying the last example in Ex.~\ref{ex:intersection_closed}).

\begin{corollary}
\label{Rn_conv_d}
The family $\mathcal{F}$ containing all convex intervals in $\R^n$ is intersection-closed and closed under the action of the standard flow (Ex.~\ref{ex:R_flow}).
\end{corollary}

\subsection{Geometric characterization of interleavings for interval modules}

Dey et al.~\cite{dey2018computing} established a necessary and sufficient criterion for interleavings of discretely presented interval $\R^n$-modules. In this section, we present an analogous characterization for interleavings of arbitrary interval $\PP$-modules. The results are stated under the assumption that $I$ and $J$ are intervals such that $I \cap J$ is also an interval. In particular, the results apply if $I,J \in \mathcal{F}$ for some family $\mathcal{F}$ of intersection-closed intervals which is closed under the flow $\Omega$.

\begin{definition}[$\e$-Trivial and Significant Intervals]
Let $\e\geq0$.
An interval $I$ in $\PP$ is called \emph{$\e$-trivial} if the $\e$-transition morphism $\transition{\Cfunc(I)}{\e}:\Cfunc(I)\to \Cfunc(I(\e))$ is the zero morphism. 
If $I$ is not $\e$-trivial, then it is called \emph{$\e$-significant}.
\end{definition}

\begin{lemma}
\label{lem:geometric condition for left interleaving}
Let $\e\geq0$ and let $I,J \subset \PP$ be intervals such that $I \cap J$ is an interval. 
Then $(\Cfunc(I),\Cfunc(J))$ are left $\e$-interleaved if and only if either $I$ is $2\e$-trivial, or the following three conditions hold:
\begin{enumerate}
    \item $I\cap J(\e)$ is $(I,J(\e))$-valid (equivalently, there exists a nonzero morphism $C(I) \to C(J(\e))$),
     \item $J\cap I(\e)$ is $(J,I(\e))$-valid (equivalently, there exists a nonzero morphism $C(J) \to C(I(\e))$), and
     \item 
     $\varnothing\neq I\cap I(2\e)\subset J(\e)$.
     \end{enumerate}
\end{lemma}

\begin{proof}
    We first note that the equivalence statements in conditions 1 and 2 follow by Corollary \ref{cor:nonzeromorphismiffvalid}.
    
    Assume that $\Cfunc(I),\Cfunc(J)$ are left $\e$-interleaved.
    Then there exist morphisms 
    $f:\Cfunc(I)\to \Cfunc(J(\e))$ and $g:\Cfunc(J) \to \Cfunc(I(\e))$ satisfying
    $g(\e)\circ f=\transition{\Cfunc(I)}{2\e}$. 
    If $I$ is $2\e$-trivial, then we are done.
    So, suppose on the other hand that $I$ is not $2\e$-trivial, i.e., the transition morphism $\transition{\Cfunc(I)}{2\e}$ is nonzero. 
    Hence $g(\e)\circ f$ is nonzero, and consequently both $f$ and $g(\e)$ are nonzero, which gives the first two conditions, and the last condition follows by Prop.~\ref{prop:nonzero composition}. 
    
    To prove the converse, first note that if $I$ is $2\e$-trivial then zero maps define a left $\e$-interleaving. Suppose, on the other hand, that $I$ is $2\e$-significant and that $I,J$ satisfy conditions $1-3$. Let $f:C(I) \to C(J(\e))$ and $g:C(J) \to C(I(\e))$ be the morphisms which are guaranteed by the first two conditions.  
    Then, the conclusion follows by applying Prop.~\ref{prop:nonzero composition} to the morphisms $f$ and $g(\e)$, noting that these morphisms can be chosen each with scalar $1_\K$, so that the composition is $g(\e) \circ f$ coincides with the morphism $\varphi_{C(I)}^{2\e}$.
\end{proof}
As a direct consequence of Lem.~\ref{lem:geometric condition for left interleaving}, we derive the following criterion.
\begin{proposition}
\label{prop:geometric condition for interleaving}
Let $\e\geq0$ and let $I,J\subset \PP$ be $2\e$-significant intervals such that $I \cap J$ is an interval. Then $\Cfunc(I),\Cfunc(J)$ are $\e$-interleaved if and only if 
the following four conditions hold:
\begin{enumerate}
   \item $I\cap J(\e)$ is $(I,J(\e))$-valid (equivalently, there exists a nonzero morphism $C(I) \to C(J(\e))$),
     \item $J\cap I(\e)$ is $(J,I(\e))$-valid (equivalently, there exists a nonzero morphism $C(J) \to C(I(\e))$), 
         \item $\varnothing \neq I\cap I(2\e)\subset J(\e)$, 
         \item $\varnothing \neq J\cap J(2\e)\subset I(\e)$.
     \end{enumerate}
\end{proposition}

\begin{proof}
Follows directly by \cref{lem:geometric condition for left interleaving} and observing that $\Cfunc(I),\Cfunc(J)$ are $\e$-interleaved if and only if the pairs $(\Cfunc(I),\Cfunc(J))$ and $(\Cfunc(J),\Cfunc(I))$ are left $\e$-interleaved.
\end{proof}

Assuming further significance on one of the intervals, interleavings of interval modules can be detected from just one-sided interleavings. The next proposition provides a crucial ingredient for the proof of our main theorem in the following section.

\begin{proposition}
\label{prop:convex-symmetry}
Let $\e\geq0$ and let $I,J\in \FF$.
If $I$ is $4\e$-significant and $\Cfunc(I),\Cfunc(J)$ are left $\e$-interleaved, then  $\Cfunc(I),\Cfunc(J)$ are $\e$-interleaved. 
\end{proposition}

\begin{proof}
Assume that $I$ is $4\e$-significant and $\Cfunc(I),\Cfunc(J)$ are left $\e$-interleaved via a left interleaving $(f,g)$. If $J$ is $2\e$-trivial, then trivially, $\Cfunc(I),\Cfunc(J)$ are right $\e$-interleaved, and we are done. Otherwise, $J$ is $2\e$-significant. By the left $\e$-interleaving we obtain  $g(\e)\circ f=\varphi_{\Cfunc(I)}^{2\e}$, hence
\[
g(3\e)\circ f(2\e)\circ g(\e)\circ f=\varphi_{\Cfunc(I)}^{2\e}(2\e)\circ\varphi_{\Cfunc(I)}^{2\e}=\varphi_{\Cfunc(I)}^{4\e} \neq \vZero,
\]
since $I$ is $4\e$-significant.
It follows that the middle composition of the left-hand side, namely
$f(2\e) \circ g(\e)$, is nonzero.
This implies that $\varnothing\neq J(\e)\cap J(3\e)\subset I(2\e)$, by \cref{prop:nonzero composition}. 
Thus, after shifting, we have $\varnothing\neq J\cap J(2\e)\subset I(\e)$.
Since $(\Cfunc(I),\Cfunc(J))$ are left $\e$-interleaved and $I$ is not $2\e$-trivial, by \cref{lem:geometric condition for left interleaving}, we have that $\varnothing\neq I\cap I(2\e)\subset J(\e)$, and $f$ and $g$ being nonzero morphisms.
Thus, $I,J$ are $2\e$-significant and the four conditions in \cref{prop:geometric condition for interleaving} are satisfied, and hence $\Cfunc(I)$ and $\Cfunc(J)$ are $\e$-interleaved.
\end{proof}

\section{Hausdorff stability of barcodes over posets}
\label{sec:Hausdorff Stability of Barcodes}

This section contains our main result, that Hausdorff distance is stable with respect to interleaving distance, together with some surrounding discussion. We begin with an example showing that the result is nontrivial.

\subsection{Instability of bottleneck and Hausdorff distances}

As shown by Botnan and Lesnick~\cite{botnan2018algebraic}, the bottleneck distance
can be arbitrarily larger than the interleaving distance; that is, the ratio
\( d_{\mathrm{B}}(M,N)/ d_{\mathrm{I}}(M,N) \) can become arbitrarily large.
This means that the barcodes are not interleaving-stable invariants of
interval–decomposable modules with respect to the bottleneck distance
\( d_{\mathrm{B}} \). In the following example, we revisit the same example used by Botnan and Lesnick
to show that the barcodes of interval–decomposable modules are also not stable
with respect to the Hausdorff distance \( d_{\mathrm{H}} \).

\begin{example}
    Consider the modules $M = C(I) \oplus C(J)$ and $N=C(K)$, shown in Figure \ref{fig:lesnick_example}(a), including a variable parameter $a \geq 1$. It was shown in \cite{botnan2018algebraic} that $d_{\mathrm I}(M,N) \leq 1$. We claim that $d_{\mathrm H}(M,N) \geq a/2$, so that taking $a \to \infty$ establishes instability of $d_{\mathrm H}$. 

    For $\e < a/2$, we claim that there exists no $\e$-correspondence $\sigma \subset B(M) \times B(N)$. Since $I$ and $J$ are $a$-significant, neither $C(I)$ nor $C(J)$ is $a/2$-interleaved with the zero module. Hence any potential $\e$-correspondence must contain, say, $(I,K)$. But $C(I)$ and $C(K)$ are not $\e$-interleaved. To see this, consider $K(\e) \cap K(-\e)$, as shown in Figure \ref{fig:lesnick_example}(b). We see that $K(\e) \cap K(-\e)$ is not contained in $I$, or equivalently, $K \cap K(2\e) \not \subset I(\e)$. This violates the interleaving condition in Prop.~\ref{prop:geometric condition for interleaving}, so the claim is proved.
\end{example}
The previous example naturally raises the question of whether the Hausdorff distance becomes
stable with respect to the interleaving distance when we restrict attention to
families of intervals that are closed under pairwise intersections.
We address this question in the next section.

 \begin{figure}[h!]
    \centering
     \begin{subfigure}{.32\textwidth}
      \begin{overpic}[unit=1mm,scale=.25]{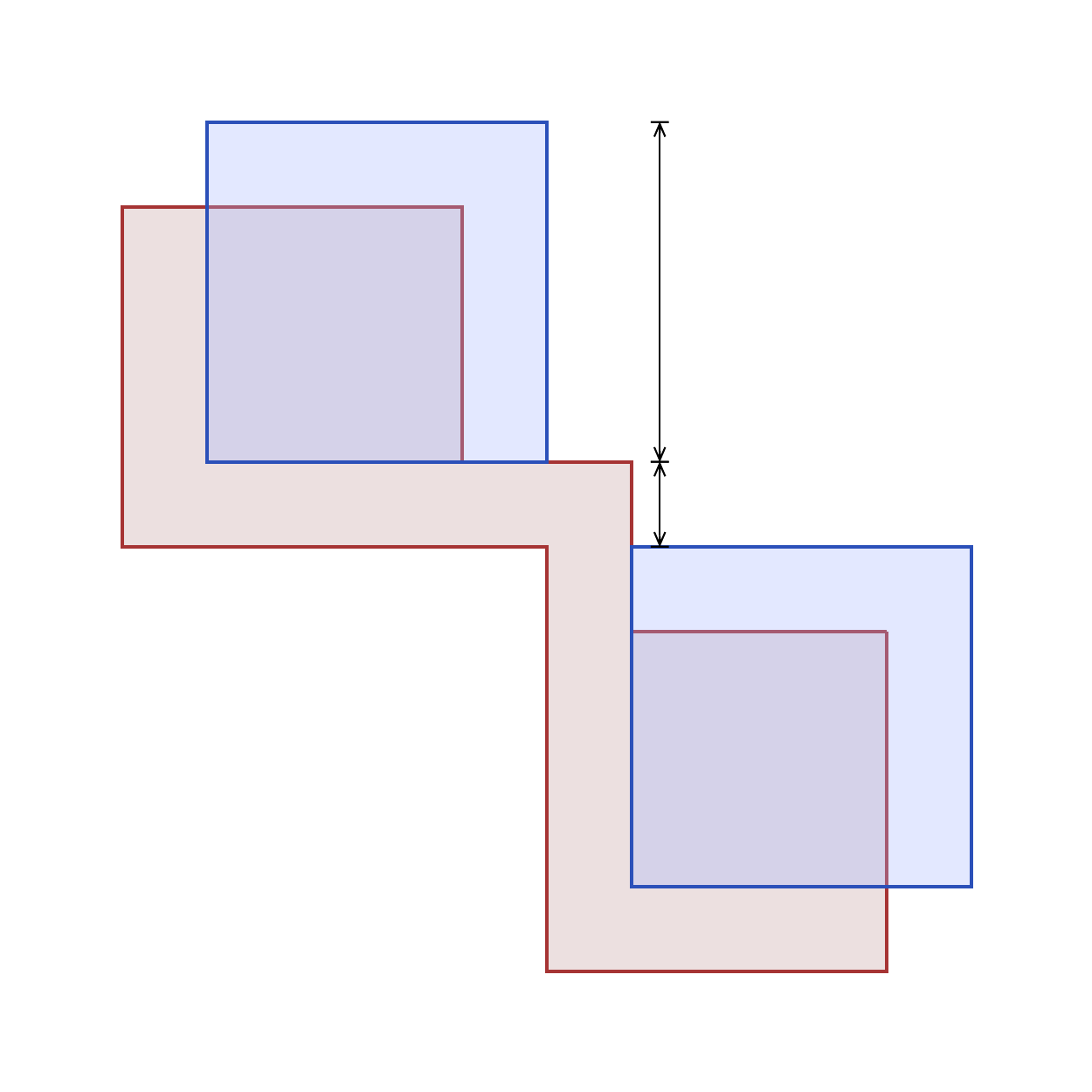}
      \put(46,83){\small $I$}
      \put(85,44){\small $J$}
      \put(52.5,52){\small $K$}
      \put(62,52){\small $1$}
      \put(62,72){\small $a$}
      \end{overpic}
      \caption{}
      \label{fig:lesnick_1}
    \end{subfigure}%
         \begin{subfigure}{.32\textwidth}
      \begin{overpic}[unit=1mm,scale=.25]{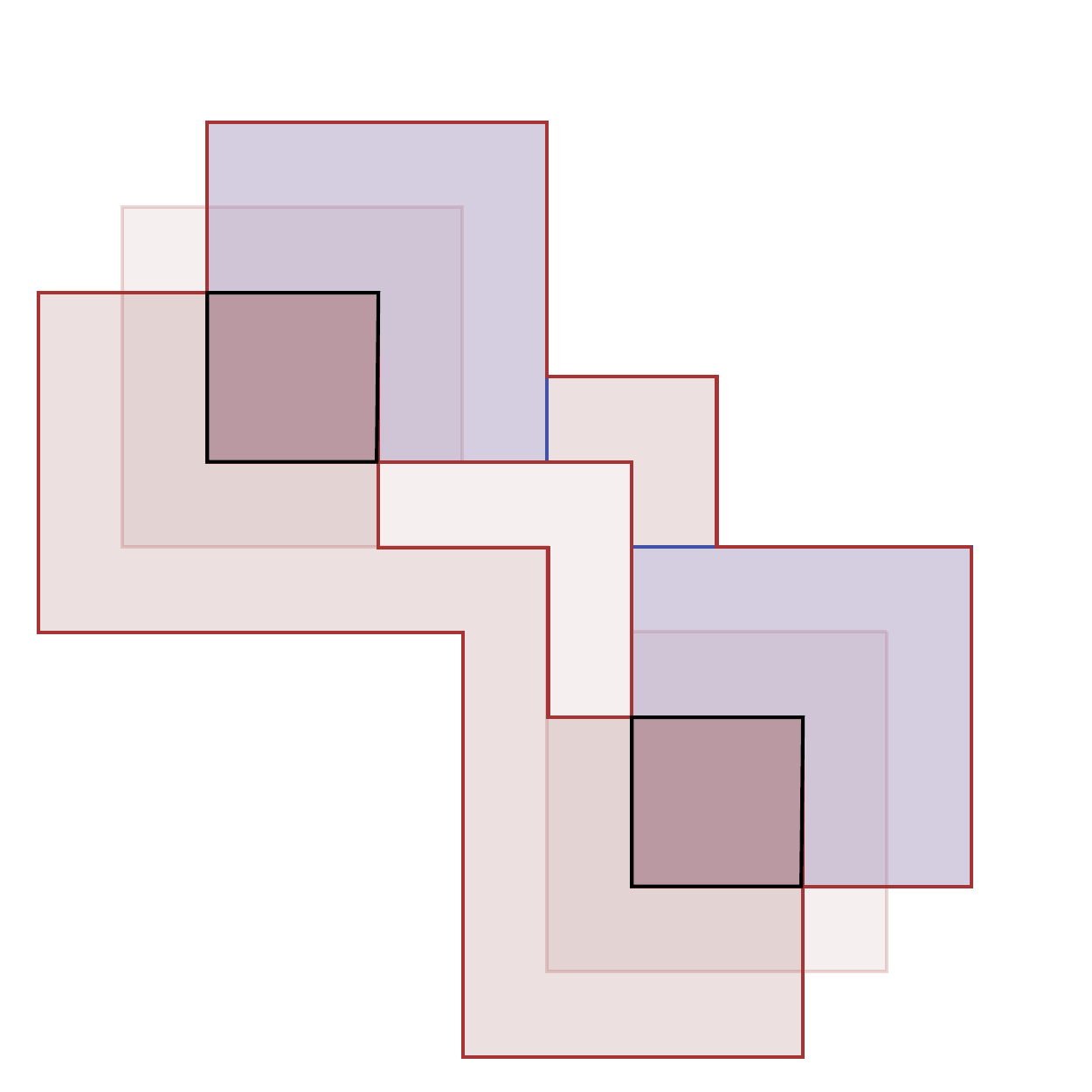}
      \end{overpic}
      \caption{}
      \label{fig:lesnick_2}
    \end{subfigure}%
      \begin{subfigure}{.32\textwidth}
      \begin{overpic}[unit=1mm,scale=.25]{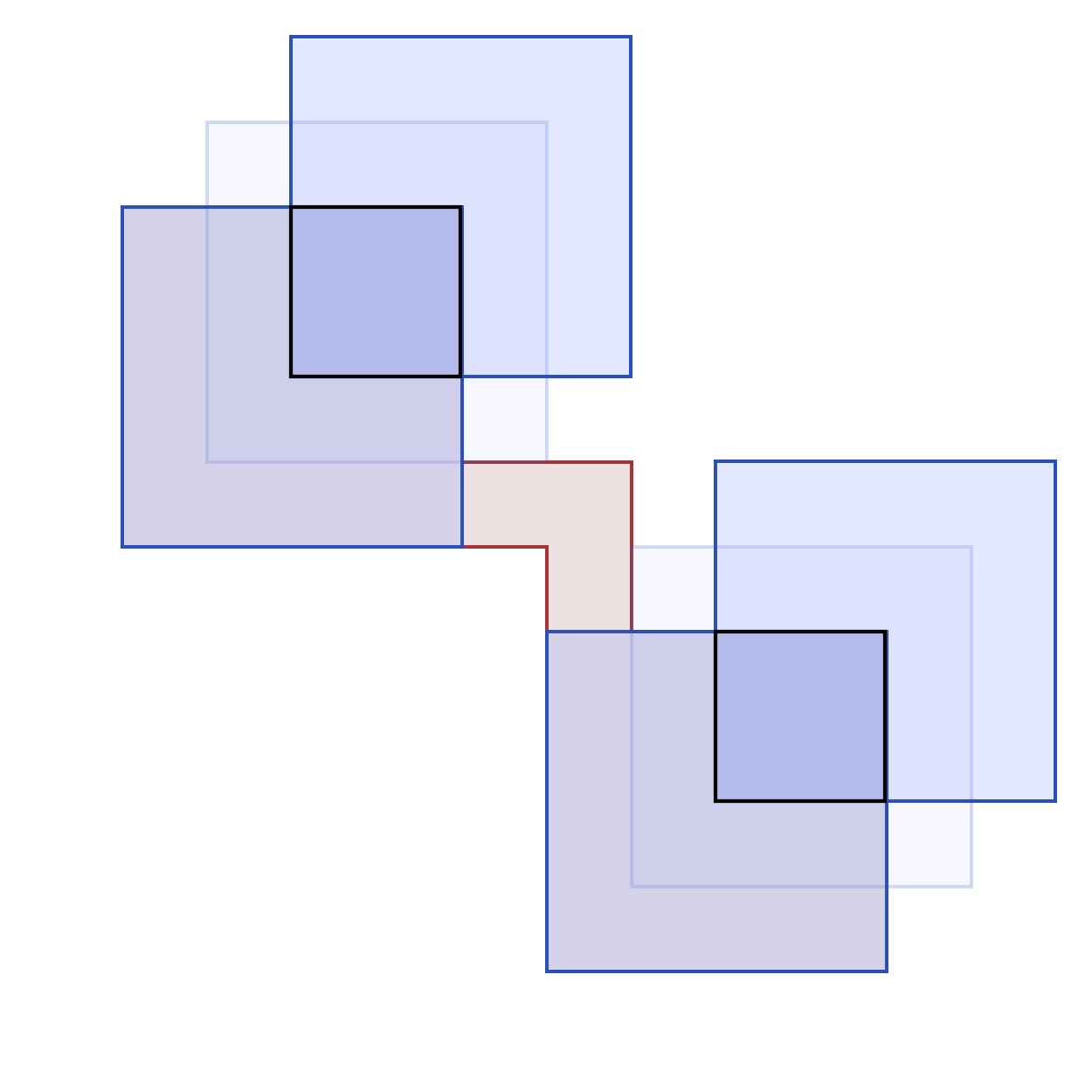}
      \end{overpic}
      \caption{}
      \label{fig:lesnick_3}
    \end{subfigure}%
    \caption{Intervals illustrating Hausdorff instability; see text for details. \textbf{(a)} Original intervals, generating persistence modules $\Cfunc(I) \oplus \Cfunc(J)$ and $\Cfunc(K)$.  \textbf{(b)} Interval $K$ shifted by $\pm \varepsilon$. The unshifted $K$ is shown with reduced opacity and the intersection $K(\varepsilon) \cap K(-\varepsilon)$ is shown with increased opacity. \textbf{(c)} Intervals $I$ and $J$ each shifted by $\pm \e$. The unshifted $I$ and $J$ are shown with reduced opacity and the intersection of the shifted versions is shown with increased opacity.}
     \label{fig:lesnick_example}
 \end{figure}
 
\subsection{Hausdorff stability of barcodes with intersection-closed intervals}

We now prove our main result, which we recall below.

\begin{theorem}
\label{thm:stability}
Let $(\PP,\Omega)$ be a poset equipped with a $\R$-flow and let $\FF$ be a family of intervals in $\PP$ that contains the empty interval, is closed under intersections, and is closed under the action of $\Omega$. 
Let $\Mfunc,\Nfunc:\PP\to\Vect_K$ be two interval-decomposable $\PP$-modules with $\barc(\Mfunc),\barc(\Nfunc)\subset \FF$. 
Then
\[d_\mathrm{H}(\Mfunc,\Nfunc)\leq 2\cdot d_\mathrm{I}(\Mfunc,\Nfunc).\]
\end{theorem}

\begin{proof}
Let $f:\Mfunc\to\Nfunc(\e)$ and $g:\Nfunc\to\Mfunc(\e)$ be an $\e$-interleaving of $M,N$.
By writing $f,g,\varphi^{2\e}_\Mfunc, \varphi^{2\e}_{\Nfunc}$ in a matrix form (see \cref{eqn:morphism_decomposition,eq:transition matrix form}) and by \cref{lem:invariant properties}, we obtain that the diagrams
\begin{equation*}
	\begin{tikzcd}
	\underset{I\in  \barc(\Mfunc)}{\bigoplus}\Cfunc(I)
	\arrow[dr,  "{\bigoplus\delta_{I,I'}\cdot\transition{\Cfunc(I)}{2\e}}"']
	\arrow[r,   swap, "{\bigoplus f_{I,J(\e)}}"']
	\&
	\underset{J\in  \barc(\Nfunc)}{\bigoplus}\Cfunc(J(\e))
	\arrow[d,  crossing over, swap, "{\bigoplus (g(\e))_{J(\e),I'(2\e)}}"']
	\&
	\&
	\underset{J\in  \barc(\Nfunc)}{\bigoplus}\Cfunc(J)
	\arrow[r,  swap, "{\bigoplus g_{J,I(\e)}}"']
	\arrow[dr,    "{\bigoplus\delta_{J,J'}\cdot\transition{\Cfunc(J)}{2\e}}"']
	\&
	\underset{I\in  \barc(\Mfunc)}{\bigoplus}\Cfunc(I(\e))
	\arrow[d,   "{\bigoplus (f(\e))_{I(\e),J'(2\e)}}"]
	\\
	\&
	\underset{I\in  \barc(\Mfunc)}{\bigoplus}\Cfunc(I(2\e))
	\&
	\&
	\&
	\underset{J\in  \barc(\Nfunc)}{\bigoplus}C(J(2\e))
	\end{tikzcd}
	\end{equation*}
	commute. 
    By \cref{eqn:composition morphism}, the left diagram yields:
\begin{equation}
    \label{eqn:I}
\delta_{I,I'}\cdot\transition{\Cfunc(I)}{2\e}=\sum_{J\in  \barc(\Nfunc)}(g(\e))_{J(\e),I'(2\e)}\circ f_{I,J(\e)},
\end{equation}
for any pair $(I,I')\in\barc(M)\times \barc(M)$.

To prove the theorem, it suffices to show that for every interval 
\( I \in \barc(\Mfunc)\coprod\{\varnothing\}\) there exists an interval 
\( J \in \barc(\Nfunc)\coprod\{\varnothing\}\) such that 
\(\Cfunc(I)\) and \(\Cfunc(J)\) are \(2\e\)-interleaved.  
The reverse direction, namely that for every \(J\) there is such an \(I\), is shown analogously.

Let \( I \in \barc(\Mfunc) \). First observe that if $I$ is $4\e$-trivial (i.e., \( \transition{\Cfunc(I)}{4\e} = \mathbf{0} \)), then \(\Cfunc(I)\) is \(2\e\)-interleaved with $\Cfunc(\emptyset)$. We therefore focus on the case where \( \transition{\Cfunc(I)}{4\e} \neq \mathbf{0} \).  
Then, in particular, \( \transition{\Cfunc(I)}{2\e} \neq \mathbf{0} \).
Applying \cref{eqn:I} with \( I' = I \), we obtain an interval 
\( J \in \barc(\Nfunc) \) such that $(g(\e))_{J(\e),\,I(2\e)} \circ f_{I,\,J(\e)} \neq \mathbf{0}$.
By \cref{prop:nonzero composition}, this implies 
\[
\emptyset \neq I \cap I(2\e) \subset J(\e).
\]
Since the composition is nonzero, both morphisms 
\( f_{I,\,J(\e)} \) and \( (g(\e))_{J(\e),\,I(2\e)} \) 
are themselves nonzero.  Hence, by \cref{lem:geometric condition for left interleaving},
\(\Cfunc(I)\) is left \(\e\)-interleaved with \(\Cfunc(J)\).
Since \(I\) is \(4\e\)-significant, \cref{prop:convex-symmetry}
upgrades this to a full \(\e\)-interleaving of 
\(\Cfunc(I)\) and \(\Cfunc(J)\), and therefore, in particular, they are 
\(2\e\)-interleaved.
\end{proof}
\begin{remark}
    Thm.~\ref{thm:stability} can be applied to certain types of interval-decomposable $\R^n$-modules: 
    \begin{enumerate}
        \item As we have seen in Cor.~\ref{Rn_conv_d}, the set of all convex intervals of $\R^n$ together with the empty interval form an intersection-closed family of intervals of $\R^n$. If $M$ and $N$ are convex interval-decomposable persistence modules over $\R^n$, then $d_{\mathrm H}(M,N)\leq 2\cdot d_{\mathrm I}(M,N)$. This is closely related to the conjecture proposed by Botnan and Lesnick~\cite{botnan2018algebraic}, according to which $d_{\mathrm B}(M,N)\leq c_{n}\cdot d_{\mathrm I}(M,N)$, with the Lipschitz constant potential depending on dimension $n$.
        \item Also, the set of all uppersets of $\R^n$ together with the empty interval form an intersection-closed family of intervals of $\R^n$.  If $M$ and $N$ are upperset interval-decomposable persistence modules over $\R^n$, then $d_{\mathrm H}(M,N)\leq 2\cdot d_{\mathrm I}(M,N)$. This is closely related to the conjecture proposed by Bjerkevik~\cite{bjerkevik2025stabilizing}, according to which $d_{\mathrm B}(M,N)\leq c_{n}\cdot d_{\mathrm I}(M,N)$. 
    \end{enumerate}  
\end{remark}
\subsection{Tightness of the bound}\label{sec:tightness}

We now construct, for any $\delta \in (0, 2)$, convex interval-decomposable modules $\Mfunc$ and $\Nfunc$ such that $d_\mathrm{H}(\Mfunc,\Nfunc) \leq (2-\delta) d_\mathrm{I}(\Mfunc,\Nfunc)$, thereby illustrating the tightness of the bound in \cref{thm:stability}. 
The construction is based on three intervals $I$, $J$, and $K$, from which we
define
\[
\Mfunc \coloneqq \Cfunc(I) \oplus \Cfunc(J),
\qquad
\Nfunc \coloneqq \Cfunc(K).
\]
We choose $I$, $J$, and $K$ so that $d_\mathrm{I}(\Mfunc,\Nfunc)=1$ while
$d_\mathrm{H}(\Mfunc,\Nfunc)=2-\delta$.  
In the visualization in \cref{fig:counterexample_3}, the intervals $I$ and $J$
appear in blue and $K$ appears in red.  
We now describe the construction.\\

Begin with a square interval $I \subset \R^2$ of side length $4-2\delta$ (so that the square is $(2-\delta)$-significant). 
Label the four corners of $I$ as follows:
\begin{center}
    $\ell_-$ (lower left), \quad 
    $\ell_+$ (upper left), \quad
    $r_-$ (lower right), \quad
    $r_+$ (upper right).
\end{center}
\cref{fig:counterexample_1} shows the interval $I$ together with its diagonal translates $I(1)$ and $I(-1)$, obtained by shifting $I$ in the $(-1,-1)$ and $(1,1)$ directions, respectively.

We construct a second interval $K$ as a 6-sided polygon such that $ I(1)\cap I(-1) \subset K$, $K$ is $2$-significant.
For any point $p\in\R^2$ and $\e\in \R$, we write $p(\e)=p-\e(1,1)$.
Three vertices of $K$, corresponding to the three red dots 
in \cref{fig:counterexample_1}, are fixed as follows:
\[
\ell_-(-1), \qquad r_+(1), \qquad 
q := \text{the upper-left corner of } I(1)\cap I(-1).
\]
To determine the remaining three vertices, we first draw the two solid rays shown in
\cref{fig:counterexample_1}:
\begin{itemize}
    \item the ray starting at $\ell_-(-1)$ and passing through $r_-(1)$,
    \item the ray starting at $r_+(1)$ and passing through $r_-(-1)$.
\end{itemize}
Since $\delta>0$, each solid ray forms a positive angle with the corresponding anti-diagonal of $I(1)$ and $I(-1)$ (shown as dotted lines in \cref{fig:counterexample_1}), and therefore the two rays diverge.  We may therefore place two vertices of $K$ far out along these rays (taken at equal
distances from their initial points).  The sixth vertex of $K$ is then placed on the line joining $q$ and $r_-$ so that the two edges incident to this vertex meet at a right angle; see \cref{fig:counterexample_2}.  
Placing the ray vertices sufficiently far away ensures that these two edges have length at least $4$, making $K$ a $2$-significant interval.

We construct one final interval $J$ as $J \coloneqq K(1)\cap K(-1)$. By construction, $I$ and $J$ are disjoint---see \cref{fig:counterexample_3} for an example where $\delta$ is relatively large and \cref{fig:counterexample_4} for an example where $\delta$ is small. 

Recall that $\Mfunc = \Cfunc(I) \oplus \Cfunc(J)$ and $\Nfunc = \Cfunc(K)$. 
The above construction ensures:
\begin{itemize}
    \item Using Prop.~\ref{prop:equivalent condition}, one can show that $d_{\mathrm I}(M,N) \geq 1$. On the other hand, there exists a $1$-interleaving of $M$ and $N$, so that $d_{\mathrm I}(M,N) = 1$.
    \item Prop.~\ref{prop:geometric condition for interleaving} implies that $d_\mathrm{I}(\Cfunc(I),\Cfunc(K)) = 2-\delta$ and $d_\mathrm{I}(\Cfunc(J),\Cfunc(K)) = 1$.
    \item Using the previous calculations, we deduce that $d_\mathrm{H}(M,N) = 2-\delta$ when $\delta \in (0,1]$. 
\end{itemize}
This proves that the bound in \cref{thm:stability} is tight. Observe that when $\delta = 0$, the rays used in the construction of $K$ are parallel, so that the construction of $K$ is no longer valid; i.e., this example cannot be extended past the theoretical Lipschitz bound of $2$. 

 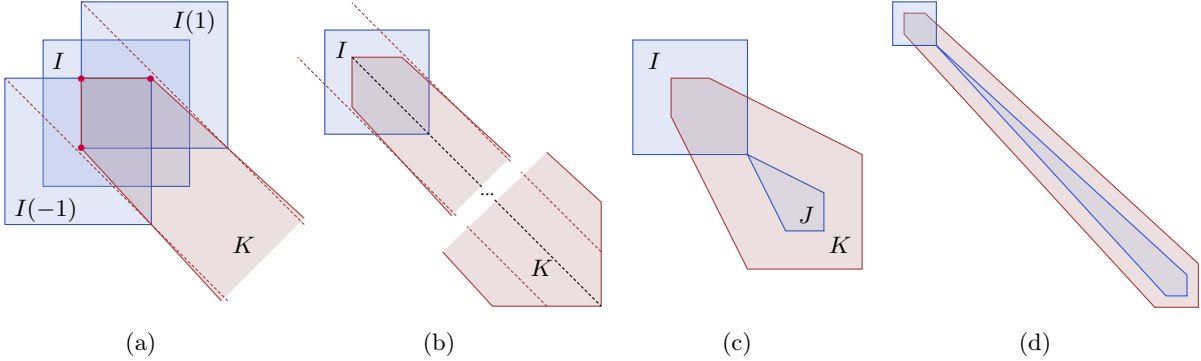
\begin{figure}[h!]
    \centering
     \begin{subfigure}{.24\textwidth}
      \centering
      \begin{overpic}[abs,unit=1mm,scale=.2]{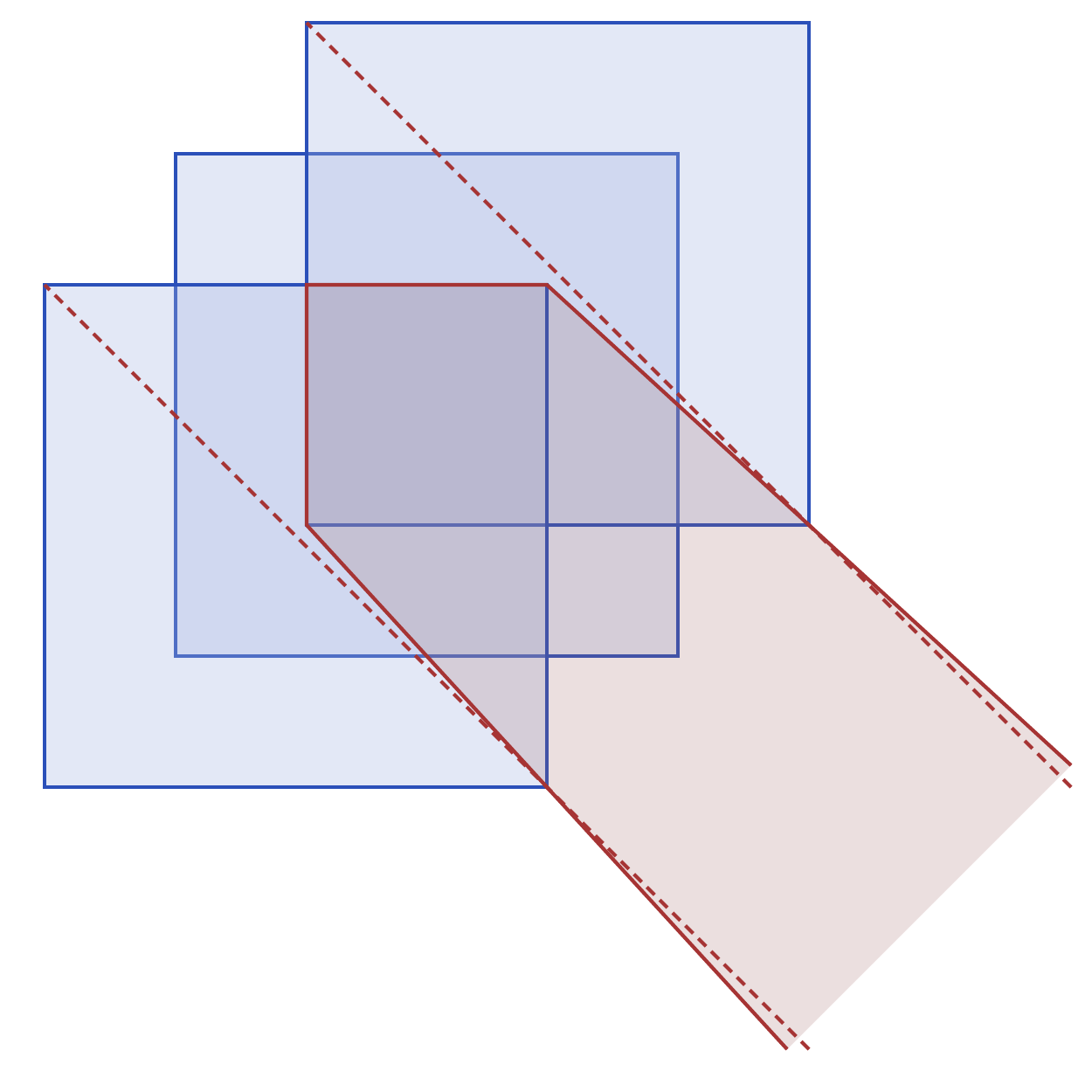}
      \put(8,32.5){\small $I$}
      \put(3,13){\small $I(-1)$}
      \put(24,38){\small $I(1)$}
      \put(32,8){\small $K$}
      \put(11.25,30.75){\tiny \color{red!10!purple} $\bullet$}
      \put(20.45,30.75){\tiny \color{red!10!purple} $\bullet$}
      \put(11.25,21.55){\tiny \color{red!10!purple} $\bullet$}
      \end{overpic}
      \caption{}
      \label{fig:counterexample_1}
    \end{subfigure}%
    \begin{subfigure}{.24\textwidth}
      \centering
      \begin{overpic}[abs,unit=1mm,scale=.2]{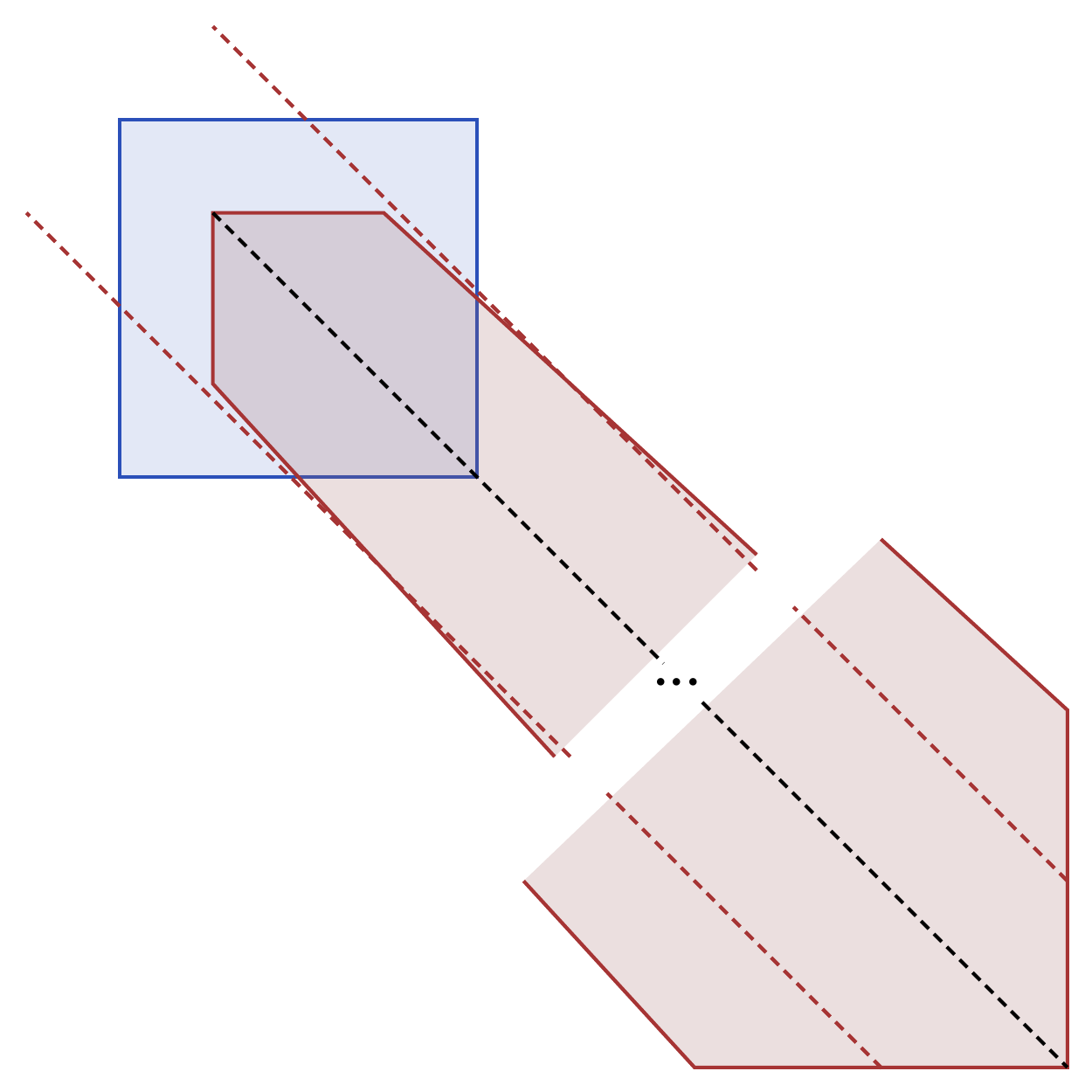}
      \put(6,34){\small $I$}
      \put(32,5){\small $K$}
      \end{overpic}
      \caption{}
      \label{fig:counterexample_2}
    \end{subfigure}%
    \begin{subfigure}{.24\textwidth}
      \centering
      \begin{overpic}[abs,unit=1mm,scale=.2]{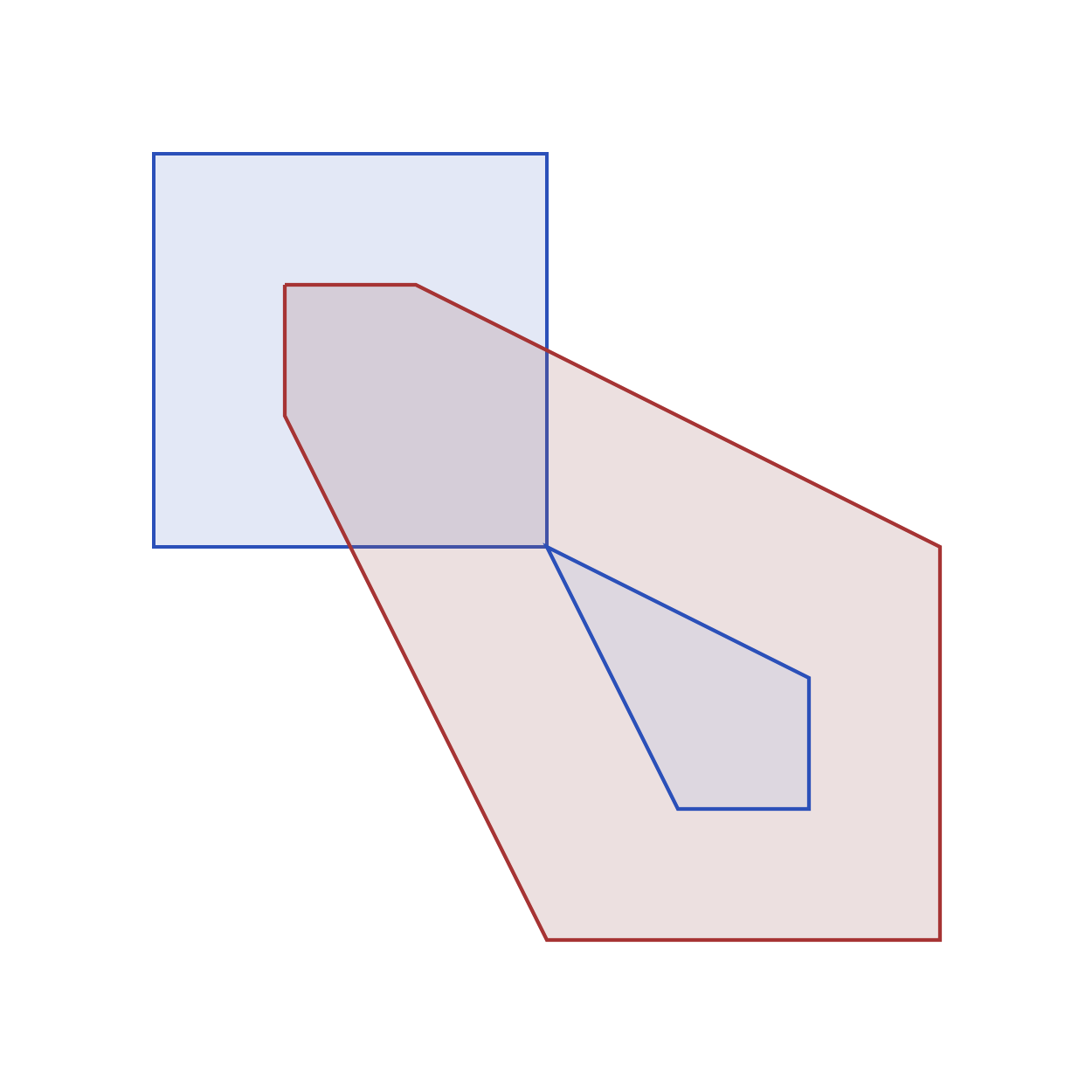}
      \put(8,32.5){\small $I$}
      \put(28,12){\small $J$}
      \put(32,8){\small $K$}
      \end{overpic}
      \caption{}
      \label{fig:counterexample_3}
    \end{subfigure}%
    \begin{subfigure}{.24\textwidth}
      \centering
      \begin{overpic}[abs,unit=1mm,scale=.2]{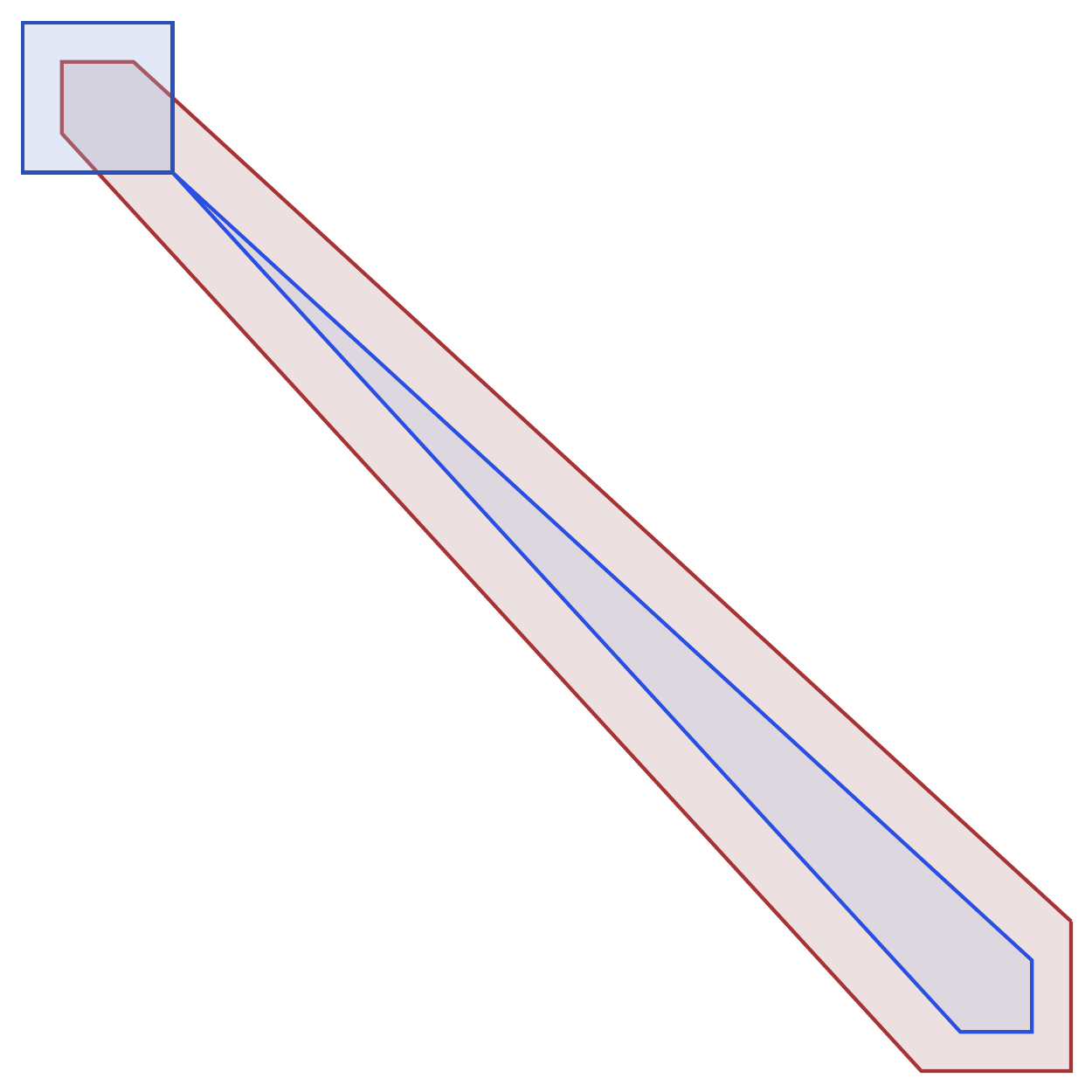}
      \end{overpic}
      \caption{}
      \label{fig:counterexample_4}
    \end{subfigure}%
    \caption{Convex intervals whose characteristic modules illustrate that the bound in \cref{thm:stability} is tight. Figures (a) and (b) illustrate the construction of the example for arbitrary $\delta$---see \cref{sec:tightness}. Figures (c) and (d) show examples of the full construction for large $\delta$ and small $\delta$, respectively.}
     \label{fig:counterexample}
 \end{figure}

 \section*{Funding Declaration}
TN was supported by NSF grants DMS--2324962 and CIF--2526630. LZ was supported by an AMS-Simons Travel Grant.

\section*{Author Contribution}
The Hausdorff stability result was obtained by AS, while the tightness of the bound was proved by TN. The technical lemmas and propositions were developed jointly by AS and TN. The coauthors LZ and MA, who joined the project at a later stage, contributed through helpful discussions, careful reading and editing of the manuscript, examples, remarks, and critical feedback on technical issues.

\bibliography{bibliography}

@article{bakke2021stability,
  title={On the stability of interval decomposable persistence modules},
  author={Bakke Bjerkevik, H{\aa}vard},
  journal={Discrete \& Computational Geometry},
  volume={66},
  number={1},
  pages={92--121},
  year={2021},
  publisher={Springer}
}

@article{botnan2022introduction,
  title={An introduction to multiparameter persistence},
  author={Botnan, Magnus Bakke and Lesnick, Michael},
  journal={arXiv preprint arXiv:2203.14289},
  volume={2},
  number={5},
  year={2022},
  publisher={CoRR}
}

@article{frosini2001size,
  title={Size functions and formal series},
  author={Frosini, Patrizio and Landi, Claudia},
  journal={Applicable Algebra in Engineering, Communication and Computing},
  volume={12},
  number={4},
  pages={327--349},
  year={2001},
  publisher={Springer}
}

@book{hausdorff1978grundzuge,
  title={Grundzuge der mengenlehre},
  author={Hausdorff, Felix},
  volume={61},
  year={1978},
  publisher={American Mathematical Soc.}
}

@article{mcfaddin2026interleaving,
  title={Interleaving distances, monoidal actions and 2-categories},
  author={McFaddin, Patrick K and Needham, Tom},
  journal={Algebraic \& Geometric Topology},
  volume={26},
  number={1},
  pages={227--281},
  year={2026},
  publisher={Mathematical Sciences Publishers}
}

@article{memoli2012some,
  title={Some properties of Gromov--Hausdorff distances},
  author={M{\'e}moli, Facundo},
  journal={Discrete \& Computational Geometry},
  volume={48},
  number={2},
  pages={416--440},
  year={2012},
  publisher={Springer}
}

@article{kim2024interleaving,
  title={Interleaving by parts: Join decompositions of interleavings and join-assemblage of geodesics},
  author={Kim, Woojin and M{\'e}moli, Facundo and Stefanou, Anastasios},
  journal={Order},
  volume={41},
  number={2},
  pages={497--537},
  year={2024},
  publisher={Springer}
}

@inproceedings{bauer2020persistence,
  title={Persistence diagrams as diagrams: A categorification of the stability theorem},
  author={Bauer, Ulrich and Lesnick, Michael},
  booktitle={Topological Data Analysis: The Abel Symposium 2018},
  pages={67--96},
  year={2020},
  organization={Springer}
}

@article{lesnick2015theory,
  title={The theory of the interleaving distance on multidimensional persistence modules},
  author={Lesnick, Michael},
  journal={Foundations of Computational Mathematics},
  volume={15},
  number={3},
  pages={613--650},
  year={2015},
  publisher={Springer}
}

@article{bubenik2015statistical,
  title={Statistical topological data analysis using persistence landscapes.},
  author={Bubenik, Peter and others},
  journal={J. Mach. Learn. Res.},
  volume={16},
  number={1},
  pages={77--102},
  year={2015}
}

@article{dey2025quasi,
  title={Quasi Zigzag Persistence: A Topological Framework for Analyzing Time-Varying Data},
  author={Dey, Tamal K and Samaga, Shreyas N},
  journal={arXiv preprint arXiv:2502.16049},
  year={2025}
}

@inproceedings{xin2023gril,
  title={GRIL: A $2 $-parameter Persistence Based Vectorization for Machine Learning},
  author={Xin, Cheng and Mukherjee, Soham and Samaga, Shreyas N and Dey, Tamal K},
  booktitle={Topological, Algebraic and Geometric Learning Workshops 2023},
  pages={313--333},
  year={2023},
  organization={PMLR}
}

@article{oudot2024stability,
  title={On the stability of multigraded Betti numbers and Hilbert functions},
  author={Oudot, Steve and Scoccola, Luis},
  journal={SIAM Journal on Applied Algebra and Geometry},
  volume={8},
  number={1},
  pages={54--88},
  year={2024},
  publisher={SIAM}
}

@article{chacholski2025koszul,
  title={Koszul complexes and relative homological algebra of functors over posets},
  author={Chach{\'o}lski, Wojciech and Guidolin, Andrea and Ren, Isaac and Scolamiero, Martina and Tombari, Francesca},
  journal={Foundations of Computational Mathematics},
  volume={25},
  number={4},
  pages={1121--1165},
  year={2025},
  publisher={Springer}
}

@article{botnan2024bottleneck,
  title={On the bottleneck stability of rank decompositions of multi-parameter persistence modules},
  author={Botnan, Magnus Bakke and Oppermann, Steffen and Oudot, Steve and Scoccola, Luis},
  journal={Advances in Mathematics},
  volume={451},
  pages={109780},
  year={2024},
  publisher={Elsevier}
}

@article{botnan2024signed,
  title={Signed barcodes for multi-parameter persistence via rank decompositions and rank-exact resolutions},
  author={Botnan, Magnus Bakke and Oppermann, Steffen and Oudot, Steve},
  journal={Foundations of Computational Mathematics},
  pages={1--60},
  year={2024},
  publisher={Springer}
}

@article{blanchette2024homological,
  title={Homological approximations in persistence theory},
  author={Blanchette, Benjamin and Br{\"u}stle, Thomas and Hanson, Eric J},
  journal={Canadian Journal of Mathematics},
  volume={76},
  number={1},
  pages={66--103},
  year={2024},
  publisher={Canadian Mathematical Society}
}

@article{asashiba2023approximation,
  title={Approximation by interval-decomposables and interval resolutions of persistence modules},
  author={Asashiba, Hideto and Escolar, Emerson G and Nakashima, Ken and Yoshiwaki, Michio},
  journal={Journal of Pure and Applied Algebra},
  volume={227},
  number={10},
  pages={107397},
  year={2023},
  publisher={Elsevier}
}

@article{kerber2018exact,
  title={Exact computation of the matching distance on 2-parameter persistence modules},
  author={Kerber, Michael and Lesnick, Michael and Oudot, Steve},
  journal={arXiv preprint arXiv:1812.09085},
  year={2018}
}

@article{cerri2013betti,
  title={Betti numbers in multidimensional persistent homology are stable functions},
  author={Cerri, Andrea and Fabio, Barbara Di and Ferri, Massimo and Frosini, Patrizio and Landi, Claudia},
  journal={Mathematical Methods in the Applied Sciences},
  volume={36},
  number={12},
  pages={1543--1557},
  year={2013},
  publisher={Wiley Online Library}
}

@article{bjerkevik2025stabilizing,
  title={Stabilizing decomposition of multiparameter persistence modules},
  author={Bjerkevik, H{\aa}vard Bakke},
  journal={Foundations of Computational Mathematics},
  pages={1--60},
  year={2025},
  publisher={Springer}
}

@book{burago2001course,
  title={A course in metric geometry},
  author={Burago, Dmitri and Burago, Yuri and Ivanov, Sergei and others},
  volume={33},
  year={2001},
  publisher={American Mathematical Society Providence}
}

@article{azumaya1950corrections,
  title={Corrections and supplementaries to my paper concerning Krull-Remak-Schmidt’s theorem},
  author={Azumaya, Gor{\^o}},
  journal={Nagoya mathematical journal},
  volume={1},
  pages={117--124},
  year={1950},
  publisher={Cambridge University Press}
}

@article{bubenik2014categorification,
  title={Categorification of persistent homology},
  author={Bubenik, Peter and Scott, Jonathan A},
  journal={Discrete \& Computational Geometry},
  volume={51},
  number={3},
  pages={600--627},
  year={2014},
  publisher={Springer}
}

@article{cohen2007stability,
  title={Stability of persistence diagrams},
  author={Cohen-Steiner, David and Edelsbrunner, Herbert and Harer, John},
  journal={Discrete \& computational geometry},
  volume={37},
  number={1},
  pages={103--120},
  year={2007},
  publisher={Springer}
}

@article{bjerkevik2016stability,
    AUTHOR = {Bjerkevik, H\aa vard Bakke},
     TITLE = {On the stability of interval decomposable persistence modules},
   JOURNAL = {Discrete Comput. Geom.},
  FJOURNAL = {Discrete \& Computational Geometry. An International Journal
              of Mathematics and Computer Science},
    VOLUME = {66},
      YEAR = {2021},
    NUMBER = {1},
     PAGES = {92--121},
      ISSN = {0179-5376},
   MRCLASS = {55N31},
  MRNUMBER = {4270636},
MRREVIEWER = {Yuichi Ike},
       DOI = {10.1007/s00454-021-00298-0},
       URL = {https://doi.org/10.1007/s00454-021-00298-0},
}

@article{scoccola2020locally,
  title={Locally Persistent Categories And Metric Properties Of Interleaving Distances},
  author={Scoccola, Luis N},
journal={Thesis},
   year={2020}
}

@Article{deSilva2018,
  author    = {{de Silva}, Vin and Munch, Elizabeth and Stefanou, Anastasios},
  title     = {Theory of interleavings on categories with a flow},
  journal   = {Theory and Applications of Categories},
  year      = {2018},
  volume    = {33},
  number    = {21},
  pages     = {583-607}
}

@article{bubenik2015metrics,
  title={Metrics for generalized persistence modules},
  author={Bubenik, Peter and De Silva, Vin and Scott, Jonathan},
  journal={Foundations of Computational Mathematics},
  volume={15},
  number={6},
  pages={1501--1531},
  year={2015},
  publisher={Springer}
}

@article{carlsson2009topology,
  title={Topology and data},
  author={Carlsson, Gunnar},
  journal={Bulletin of the American Mathematical Society},
  volume={46},
  number={2},
  pages={255--308},
  year={2009}
}

@book{edelsbrunner2010computational,
  title={Computational topology: an introduction},
  author={Edelsbrunner, Herbert and Harer, John},
  year={2010},
  publisher={American Mathematical Soc.}
}

@article{bjerkevik2019computing,
  title={Computing the interleaving distance is NP-hard},
  author={Bjerkevik, H{\aa}vard Bakke and Botnan, Magnus Bakke and Kerber, Michael},
  journal={Foundations of Computational Mathematics},
  pages={1--35},
  year={2019},
  publisher={Springer}
}

@article{botnan2018algebraic,
  title={Algebraic stability of zigzag persistence modules},
  author={Botnan, Magnus and Lesnick, Michael},
  journal={Algebraic \& geometric topology},
  volume={18},
  number={6},
  pages={3133--3204},
  year={2018},
  publisher={Mathematical Sciences Publishers}
}

@inproceedings{chazal2009proximity,
  title={Proximity of persistence modules and their diagrams},
  author={Chazal, Fr{\'e}d{\'e}ric and Cohen-Steiner, David and Glisse, Marc and Guibas, Leonidas J and Oudot, Steve Y},
  booktitle={Proceedings of the twenty-fifth annual symposium on Computational geometry},
  pages={237--246},
  year={2009}
}

@article{miller2019modules,
  title={Modules over posets: commutative and homological algebra},
  author={Miller, Ezra},
  journal={arXiv preprint arXiv:1908.09750},
  year={2019}
}

@article{dey2018computing,
  title={Computing Bottleneck Distance for Multi-parameter Interval Decomposable Persistence Modules},
  author={Dey, Tamal K and Xin, Cheng},
  journal={arXiv preprint arXiv:1803.02869},
  year={2018}
}

\end{document}